\documentclass[12pt,reqno]{amsart}%
\usepackage{amsmath, amsfonts, amssymb, amsthm, amscd, amsbsy}
\usepackage{fancyhdr}
\usepackage[usenames,dvipsnames,svgnames,x11names,hyperref]{xcolor}
\usepackage{geometry}
\usepackage{graphicx}
\usepackage{hyperref}
\usepackage{amsmath}
\usepackage{amsfonts}
\usepackage{amssymb}
\usepackage{enumerate}%
\providecommand{\U}[1]{\protect\rule{.1in}{.1in}}
\hypersetup{backref=true,
pagebackref=true,
hyperindex=true,
colorlinks=true,
breaklinks=true,
urlcolor=NavyBlue,
linkcolor=Fuchsia,
bookmarks=true,
bookmarksopen=false,
filecolor=black,
citecolor=ForestGreen,
linkbordercolor=red
}
\newtheorem{theorem}{Theorem}[section]
\theoremstyle{plain}

\newtheorem{corollary}{Corollary}[section]

\newtheorem{remark}{Remark}[section]

\numberwithin{equation}{section}
\allowdisplaybreaks
\AtBeginDocument{\hypersetup{pdfborder={0 0 0.01}}}
\begin{document}
\title[A new approach to weighted Hardy-Rellich inequalities]{A new approach to weighted Hardy-Rellich inequalities: improvements,
symmetrization principle and symmetry breaking}
\author{Anh Xuan Do}
\address{Anh Xuan Do: Department of Mathematics\\
University of Connecticut\\
Storrs, CT 06269, USA}
\email{anh.do@uconn.edu}
\author{Nguyen Lam}
\address{Nguyen Lam: School of Science and the Environment\\
Grenfell Campus, Memorial University of Newfoundland\\
Corner Brook, NL A2H5G4, Canada}
\email{nlam@grenfell.mun.ca}
\author{Guozhen Lu }
\address{Guozhen Lu: Department of Mathematics\\
University of Connecticut\\
Storrs, CT 06269, USA}
\email{guozhen.lu@uconn.edu}
\thanks{A. Do and G. Lu were partially supported by collaborations grants from the
Simons Foundation and a Simons Fellowship. N. Lam was partially supported by
an NSERC Discovery Grant.}
\date{\today}

\begin{abstract}
We investigate necessary and sufficient conditions on the weights for the
Hardy-Rellich inequalities to hold, and propose a new way to use the notion of
Bessel pair to establish the optimal Hardy-Rellich type inequalities. Our
results sharpened earlier Hardy-Rellich and Rellich type inequalities in the
literature. We also study several results about the symmetry and symmetry
breaking properties of the Rellich type and Hardy-Rellich type inequalities,
and then partially answered an open question raised by Ghoussoub and
Moradifam. Namely, we will present conditions on the weights such that the
Rellich type and Hardy-Rellich type inequalities hold for all functions if and
only if the same inequalities hold for all radial functions.

\end{abstract}
\maketitle

\section{Introduction}

The classical Hardy inequality claims that for dimensions $N\neq2:$
\begin{equation}
\int_{\mathbb{R}^{N}}\left\vert \nabla u\right\vert ^{2}dx\geq\frac{\left(
N-2\right)  ^{2}}{4}\int_{\mathbb{R}^{N}}\frac{\left\vert u\right\vert ^{2}%
}{\left\vert x\right\vert ^{2}}dx\text{, }\forall u\in C_{0}^{\infty}\left(
\mathbb{R}^{N}\right)  \text{.} \label{H}%
\end{equation}
It is one of the most frequently used inequalities in analysis, and is studied
intensively and extensively in the literature. We refer the interested reader
to, e.g., \cite{BEL, KMP2007, KP, Maz11, OK}, which are some of the standard
references on the topic.

It it worth mentioning that the constant $\frac{\left(  N-2\right)  ^{2}}{4}$
in (\ref{H}) is optimal, and that the Schr\"{o}dinger operator $-\Delta
-\frac{\left(  N-2\right)  ^{2}}{4}\frac{1}{\left\vert x\right\vert ^{2}}$ is
positive definite and critical on $\mathbb{R}^{N}$. Therefore, for instance,
there is no strictly positive $P$ in $C^{1}\left(  0,\infty\right)  $ such
that the inequality%
\[
\int_{\mathbb{R}^{N}}\left\vert \nabla u\right\vert ^{2}dx-\frac{\left(
N-2\right)  ^{2}}{4}\int_{\mathbb{R}^{N}}\frac{\left\vert u\right\vert ^{2}%
}{\left\vert x\right\vert ^{2}}dx\geq\int_{\mathbb{R}^{N}}P\left(  \left\vert
x\right\vert \right)  \left\vert u\right\vert ^{2}dx
\]
holds for all $u\in C_{0}^{\infty}\left(  \mathbb{R}^{N}\right)  $. However,
the situation is very different on bounded domains. Indeed, in the celebrated
paper of Brezis and V\'{a}zquez \cite{BV97}, the authors proved that

\begin{theorem}
[Brezis-V\'{a}zquez \cite{BV97}]\label{t0}\textit{For any bounded domain
}$\Omega\subset%
\mathbb{R}
^{N}$\textit{, }$N\geq2$\textit{, and every }$u\in H_{0}^{1}\left(
\Omega\right)  $\textit{, }%
\begin{equation}%
{\displaystyle\int\limits_{\Omega}}
\left\vert \nabla u\right\vert ^{2}dx-\left(  \frac{N-2}{2}\right)  ^{2}%
{\displaystyle\int\limits_{\Omega}}
\frac{\left\vert u\right\vert ^{2}}{\left\vert x\right\vert ^{2}}dx\geq
z_{0}^{2}\omega_{N}^{\frac{2}{N}}\left\vert \Omega\right\vert ^{-\frac{2}{N}}%
{\displaystyle\int\limits_{\Omega}}
\left\vert u\right\vert ^{2}dx \label{HBV}%
\end{equation}
\textit{where }$z_{0}=2.4048...$\textit{ is the first zero of the Bessel
function }$J_{0}\left(  z\right)  $ \textit{and} $\omega_{N}$ \textit{is the
volume of the unit ball in} $\mathbb{R}^{N}$\textit{. The constant }$z_{0}%
^{2}\omega_{N}^{\frac{2}{N}}\left\vert \Omega\right\vert ^{-\frac{2}{N}}%
$\textit{ is optimal when }$\Omega$\textit{ is a ball but is not achieved in
the Sobolev space }$H_{0}^{1}\left(  \Omega\right)  $\textit{.}
\end{theorem}

The fact that $z_{0}^{2}\omega_{N}^{\frac{2}{N}}\left\vert \Omega\right\vert
^{-\frac{2}{N}}$ is optimal but is still not achieved by nontrivial functions
led\textit{ }Brezis and V\'{a}zquez to ask whether the term on the LHS of
(\ref{HBV}) is just the first term of an infinite series of remainder terms.
Actually, Brezis and V\'{a}zquez's pioneering improvements of (\ref{H}) and
their question have received enormous attention in the literature. The
interested reader is referred to \cite{BFT04, BGGP20, BGR22, CFLL23, DPP17,
DLL22, FLL21, FLL23, FS08, LLZ19, LLZ20, Wang}, to name just a few, for some
recent developments on the field.

There are several proofs of (\ref{H}) in the literature. For instance, one of
the standard methods to prove the Hardy inequality (\ref{H}) and many other
functional and geometric inequalities is to use the classical rearrangement
argument. Indeed, by the well-known P\'{o}lya--Szeg\H{o} inequality and the
Hardy-Littlewood inequality (see, e.g. \cite{LL01}), we have that%
\begin{align*}
\int_{\mathbb{R}^{N}}\left\vert \nabla u\right\vert ^{2}dx  &  \geq
\int_{\mathbb{R}^{N}}\left\vert \nabla u^{\ast}\right\vert ^{2}dx\\
\int_{\mathbb{R}^{N}}\frac{\left\vert u^{\ast}\right\vert ^{2}}{\left\vert
x\right\vert ^{2}}dx  &  \geq\int_{\mathbb{R}^{N}}\frac{\left\vert
u\right\vert ^{2}}{\left\vert x\right\vert ^{2}}dx.
\end{align*}
Here $u^{\ast}$ is the familiar symmetric decreasing rearrangement of $u$.
(See the monograph of Stein and Weiss \cite{SteinWeiss}, Chapter V.)
Therefore, one can assume that $u$ is radially symmetric and (\ref{H}) is a
consequence of the one-dimensional Hardy inequality :%
\[
\int_{0}^{\infty}\left\vert u^{\prime}\left(  r\right)  \right\vert ^{2}%
\geq\frac{1}{4}\int_{0}^{\infty}\frac{\left\vert u\left(  r\right)
\right\vert ^{2}}{r^{2}}dr.
\]
It is also worthy to note that in general, a higher order version of the
P\'{o}lya--Szeg\H{o} principle does not hold. Therefore, it is challenging and
interesting to establish symmetrization principles for higher order functional
and geometric inequalities. One of our main goals in this paper is to set up
such symmetrization principles for higher order Hardy-Rellich inequalities and
Rellich inequalities. See Theorem \ref{T5} and Theorem \ref{T6} for more details.

In \cite{GM}, Ghoussoub and Moradifam introduced the notion of Bessel pair and
used it to unify and improve many known Hardy type inequalities in the
literature. Moreover, they also applied this concept to establish several
Hardy-Rellich type inequalities. In particular, they proved that

\begin{theorem}
[Ghoussoub-Moradifam \cite{GM}]\label{t1}\textit{Let }$0<R\leq\infty$\textit{,
}$V_{1}$ \textit{and }$W_{1}$\textit{ be positive }smooth \textit{functions on
}$\left(  0,R\right)  $. Assume that $\left(  V_{1},W_{1}\right)  $ is a
$N$-dimensional Bessel pair on $\left(  0,R\right)  $, that is, there exists a
positive function $\varphi$ such that
\[
\left(  r^{N-1}V_{1}\varphi^{\prime}\right)  ^{\prime}+r^{N-1}W_{1}%
\varphi=0\text{ on }\left(  0,R\right)  \text{.}%
\]
Then we have for all $u\in C_{0}^{\infty}\left(  B_{R}\setminus\left\{
0\right\}  \right)  $ that
\begin{equation}
\int_{B_{R}}V_{1}\left(  \left\vert x\right\vert \right)  \left\vert \nabla
u\right\vert ^{2}dx\geq\int_{B_{R}}W_{1}\left(  \left\vert x\right\vert
\right)  \left\vert u\right\vert ^{2}dx. \label{wH}%
\end{equation}
Also, if
\begin{equation}
W_{1}-\frac{2V_{1}}{r^{2}}+\frac{2V_{1}^{\prime}}{r}-V_{1}^{\prime\prime}%
\geq0\text{ on }\left(  0,R\right)  , \label{Con}%
\end{equation}
then we have for all $u\in C_{0}^{\infty}\left(  B_{R}\setminus\left\{
0\right\}  \right)  $ that
\begin{align}
\int_{B_{R}}V_{1}\left(  \left\vert x\right\vert \right)  \left\vert \Delta
u\right\vert ^{2}dx  &  \geq\int_{B_{R}}W_{1}\left(  \left\vert x\right\vert
\right)  \left\vert \nabla u\right\vert ^{2}dx\label{HR}\\
&  \hspace{0.2in}+\left(  N-1\right)  \int_{B_{R}}\left(  \frac{V_{1}\left(
\left\vert x\right\vert \right)  }{\left\vert x\right\vert ^{2}}-\frac
{V_{1}^{\prime}\left(  \left\vert x\right\vert \right)  }{\left\vert
x\right\vert }\right)  \left\vert \nabla u\right\vert ^{2}dx.\nonumber
\end{align}

\end{theorem}

One can use (\ref{Con}) and (\ref{HR}) to derive some Hardy-Rellich type
inequalities. For instance, if we choose $V_{1}=1$ and $W_{1}=\frac{\left(
N-2\right)  ^{2}}{4r^{2}}$, then
\begin{equation}
\int_{\mathbb{R}^{N}}\left\vert \Delta u\right\vert ^{2}dx\geq\frac{N^{2}}%
{4}\int_{\mathbb{R}^{N}}\frac{\left\vert \nabla u\right\vert ^{2}}{\left\vert
x\right\vert ^{2}}dx,\text{ }\forall u\in C_{0}^{\infty}\left(  \mathbb{R}%
^{N}\setminus\left\{  0\right\}  \right)  ,\text{ }N\geq5. \label{HR1}%
\end{equation}
The constant $\frac{N^{2}}{4}$ in the above Hardy-Rellich type inequality
appears as a sum of $\frac{\left(  N-2\right)  ^{2}}{4}$ in the Bessel pair
term, and $\left(  N-1\right)  $ in the second term of (\ref{HR}). Therefore,
the constant $\frac{N^{2}}{4}$ is optimal when $N\geq5$ since $\frac{\left(
N-2\right)  ^{2}}{4}$ is sharp in the $N$-dimensional Bessel pair $\left(
1,\frac{\left(  N-2\right)  ^{2}}{4r^{2}}\right)  $. See, for instance,
\cite{Bec08, TZ07, Yaf99}.

We note that the approach in \cite{GM} makes use of the spherical harmonics
decomposition. Also, without the condition (\ref{Con}), it has been showed
that (\ref{HR}) holds for all radial functions $u$. Therefore, they need to
impose the condition (\ref{Con}) to prevent the symmetry breaking in the
Hardy-Rellich inequality (\ref{HR}). We also mention here that though in
\cite{GM}, the authors mainly studied the Bessel pairs and their applications
in deriving Hardy type inequalities and Hardy-Rellich type inequalities for
positive radial potentials, there are many important Hardy type inequalities
and Hardy-Rellich type inequalities of the forms
\[
\int V\left(  |x|\right)  \left\vert \nabla u\right\vert ^{2}dx\geq\int
W\left(  |x|\right)  \left\vert u\right\vert ^{2}dx
\]
and%
\[
\int V\left(  |x|\right)  \left\vert \Delta u\right\vert ^{2}dx\geq\int
W\left(  |x|\right)  \left\vert \nabla u\right\vert ^{2}dx
\]
in which the weights $W$ are not always nonnegative. Therefore, since the
condition (\ref{Con}) is rather unusual, and because of the assumptions on the
positivity of the weights, there are several Hardy-Rellich type inequalities
that cannot be deduced from Theorem \ref{t1}. Some of the examples that we
have in mind are the second order Heisenberg uncertainty principle and the
second order Hydrogen uncertainty principle (see \cite{CFL23, CFL22, DN21}):%

\begin{align*}
\int_{\mathbb{R}^{N}}|\Delta u|^{2}dx  &  \geq\int_{\mathbb{R}^{N}}\left(
N+2-\left\vert x\right\vert ^{2}\right)  |\nabla u|^{2}dx,\\
\int_{\mathbb{R}^{N}}|\Delta u|^{2}dx  &  \geq\int_{\mathbb{R}^{N}}\left(
\dfrac{N+1}{|x|}-1\right)  |\nabla u|^{2}dx.
\end{align*}

Motivated by this remark, the first goal of this paper is to offer a new way
to employ the notion of Bessel pair to derive new sharp Hardy-Rellich type
inequalities, espcially when the weight $W$ is not always nonnegative. Our
observation is that, for instance, in the Hardy-Rellich inequality
(\ref{HR1}), since $\frac{N^{2}}{4}=\frac{\left(  N+2-2\right)  ^{2}}{4}$, we
can consider $\left(  1,\frac{N^{2}}{4r^{2}}\right)  $ as a $\left(
N+2\right)  $-dimensional Bessel pair. Therefore, in our approach, instead of
using the $N$-dimensional Bessel pair as in Theorem \ref{t1}, we will get
involved with $\left(  N+2\right)  $-dimensional Bessel pairs in deriving the
Hardy-Rellich type inequalities on $\mathbb{R}^{N}$. More precisely, our first
main result can be read as follows:

\begin{theorem}
\label{T2}\textit{Let }$N\geq1$,\textit{ }$0<R\leq\infty$\textit{, }$V_{2}%
\geq0$ and $W_{2}$\textit{ be smooth functions on }$\left(  0,R\right)  $.
Assume that $\left(  V_{2},\left(  W_{2}+N\dfrac{V_{2}^{\prime}}{r}\right)
\right)  $ is a $\left(  N+2\right)  $-dimensional Bessel pair on $\left(
0,R\right)  $. Then

\begin{enumerate}
\item for all radial functions $u\in C_{0}^{\infty}\left(  B_{R}%
\setminus\left\{  0\right\}  \right)  :$
\begin{equation}
\int_{B_{R}}V_{2}(|x|)|\Delta u|^{2}dx\geq\int_{B_{R}}W_{2}(|x|)|\nabla
u|^{2}dx \label{wHR}%
\end{equation}

\item Moreover, if%
\begin{align}
&  2\int_{0}^{R}V_{2}(r)r^{N-3}\left\vert u^{\prime}\right\vert ^{2}%
dr-\int_{0}^{R}V_{2}^{\prime\prime}\left(  r\right)  r^{N-3}\left\vert
u\right\vert ^{2}dr-\left(  N-5\right)  \int_{0}^{R}V_{2}^{\prime}\left(
r\right)  r^{N-4}\left\vert u\right\vert ^{2}dr\label{ConM}\\
&  \hspace{0.2in} +(3N-9)\int_{0}^{R}V_{2}(r)r^{N-5}\left\vert u\right\vert
^{2}dr-\int_{0}^{R}W_{2}(r)r^{N-3}\left\vert u\right\vert ^{2}dr\geq0\text{
}\forall u\in C_{0}^{\infty}\left(  0,R\right)  ,\nonumber
\end{align}
then (\ref{wHR}) holds for all $u\in C_{0}^{\infty}\left(  B_{R}%
\setminus\left\{  0\right\}  \right)  $.
\end{enumerate}
\end{theorem}

Our method of proving Theorem \ref{T2} is spherical harmonics decomposition.
Therefore, we also need to impose the condition (\ref{ConM}) to prevent
symmetry breaking phenomenon. We also note that the constraint (\ref{ConM}) is
rather artificial. Nevertheless, under the assumption that $\left(
V_{2},\left(  W_{2}+N\dfrac{V_{2}^{\prime}}{r}\right)  \right)  $ is a
$\left(  N+2\right)  $-dimensional Bessel pair on $\left(  0,R\right)  $, we
can replace (\ref{ConM}) with a much more natural condition%
\begin{equation}
V_{2}^{\prime\prime}(r)-3\dfrac{V_{2}^{\prime}(r)}{r}-\left(  N-5\right)
\dfrac{V_{2}(r)}{r^{2}}\leq0\text{ on }\left(  0,R\right)  . \label{Con2}%
\end{equation}
We note that this condition only involves the weight function $V_{2}$ rather
than the entire pair $(V_{2}, W_{2})$, which was required in \eqref{Con} of
Theorem \ref{T1}.

Indeed, by showing that condition \eqref{Con2} implies \eqref{ConM}, we have
the following

\begin{theorem}
\label{T2a}\textit{Let }$N\geq1$,\textit{ }$0<R\leq\infty$\textit{, }%
$V_{2}\geq0$ and $W_{2}$\textit{ be smooth functions on }$\left(  0,R\right)
$. Assume that $\left(  V_{2},\left(  W_{2}+N\dfrac{V_{2}^{\prime}}{r}\right)
\right)  $ is a $\left(  N+2\right)  $-dimensional Bessel pair on $\left(
0,R\right)  $. If $V_{2}(x)$ satisfies \eqref{Con2}, then for all functions
$u\in C_{0}^{\infty}\left(  B_{R}\setminus\left\{  0\right\}  \right)  :$
\begin{equation}
\label{weights}\int_{B_{R}}V_{2}(|x|)|\Delta u|^{2}dx\geq\int_{B_{R}}%
W_{2}(|x|)|\nabla u|^{2}dx.
\end{equation}

\end{theorem}

Actually, we will show that we can get the same result as the one by Ghoussoub
and Moradifam in \cite{GM} (Theorem \ref{t1}) with our condition (\ref{Con2}).
More precisely, we have

\begin{theorem}
\label{C1}\textit{Let }$N\geq1$,\textit{ }$0<R\leq\infty$\textit{, }$V_{1}%
\geq0$ and $W_{1}$\textit{ be smooth functions on }$\left(  0,R\right)  $.
Assume that $\left(  V_{1},W_{1}\right)  $ is a $N$-dimensional Bessel pair on
$\left(  0,R\right)  $. If $V_{1}(x)$ satisfies \eqref{Con2}, then for all
$u\in C_{0}^{\infty}\left(  B_{R}\setminus\left\{  0\right\}  \right)  $:
\begin{align}
\label{weights2}\int_{B_{R}}V_{1}\left(  \left\vert x\right\vert \right)
\left\vert \Delta u\right\vert ^{2}dx  &  \geq\int_{B_{R}}W_{1}\left(
\left\vert x\right\vert \right)  \left\vert \nabla u\right\vert ^{2}%
dx\nonumber\\
&  \hspace{0.2in} +\left(  N-1\right)  \int_{B_{R}}\left(  \frac{V_{1}\left(
\left\vert x\right\vert \right)  }{\left\vert x\right\vert ^{2}}-\frac
{V_{1}^{\prime}\left(  \left\vert x\right\vert \right)  }{\left\vert
x\right\vert }\right)  \left\vert \nabla u\right\vert ^{2}dx.
\end{align}

\end{theorem}

\begin{remark}
We note that unlike the condition \eqref{Con} in Theorem \ref{t1}, our
condition \eqref{Con2} in Theorem \ref{C1} only involves the weight function
$V_{1}$. Therefore, in many applications, our Theorem \ref{C1} is an
improvement of Theorem \ref{t1}, and can be used to derive several
Hardy-Rellich type inequalities that can not be deduced from Theorem \ref{t1}.
In particular, the condition (\ref{Con2}) is naturally verified in the
standard case $V_{1}=1$ and $N\geq5$ without any extra condition on the
potential $W_{1}$. Note that in this case $V_{1}=1$, the condition (\ref{Con})
in Theorem \ref{t1} becomes $W_{1}-\frac{2}{r^{2}}\geq0$ on $\left(
0,R\right)  $. This implies that the weight $W_{1}$ must be strictly positive.
However, our Theorem \eqref{C1} does not require $W_{1}$ to be nonnegative.
\end{remark}

\medskip

In this situation, we further obtain as applications of Theorem \ref{T2a} and
Theorem \ref{C1} the following result which is of its independent interest.

\begin{corollary}
\label{T1}\textit{Let }$N\geq5$,\textit{ }$0<R\leq\infty$\textit{, }$W_{2}%
$\textit{ be a smooth function on }$\left(  0,R\right)  $. Then

\begin{enumerate}
\item If $\left(  1,W_{2}\right)  $ is a $\left(  N+2\right)  $-dimensional
Bessel pair on $\left(  0,R\right)  $, then for all $u\in C_{0}^{\infty
}\left(  B_{R}\setminus\left\{  0\right\}  \right)  $:
\[
\int_{B_{R}}|\Delta u|^{2}dx\geq\int_{B_{R}}W_{2}(|x|)|\nabla u|^{2}dx.
\]

\item If $\left(  1,W_{2}\right)  $ is a $N$-dimensional Bessel pair on
$\left(  0,R\right)  $, then for all $u\in C_{0}^{\infty}\left(
B_{R}\setminus\left\{  0\right\}  \right)  $:
\begin{equation}
\int_{B_{R}}\left\vert \Delta u\right\vert ^{2}dx\geq\int_{B_{R}}W_{2}\left(
\left\vert x\right\vert \right)  \left\vert \nabla u\right\vert ^{2}dx+\left(
N-1\right)  \int_{B_{R}}\frac{\left\vert \nabla u\right\vert ^{2}}{\left\vert
x\right\vert ^{2}}dx. \label{wwHR}%
\end{equation}

\end{enumerate}
\end{corollary}

As a consequence of our main results, we can derive as many Hardy-Rellich type
inequalities as we can form Bessel pairs. For instance, we have the following
Hardy-Rellich inequalities:

\begin{corollary}
\textit{Let }$N\geq5$ and\textit{ }$R>0$. Then

\begin{itemize}
\item[a)] (Hardy-Rellich inequality) For all $u\in C_{0}^{\infty}\left(
\mathbb{R}^{N}\right)  :$%
\[
\int_{\mathbb{R}^{N}}|\Delta u|^{2}dx\geq\dfrac{N^{2}}{4}\int_{\mathbb{R}^{N}%
}\dfrac{|\nabla u|^{2}}{|x|^{2}}dx.
\]

\item[b)] For all $u\in C_{0}^{\infty}\left(  B_{R}\right)  :$%
\[
\int_{B_{R}}|\Delta u|^{2}dx\geq\frac{N^{2}}{4}\int_{B_{R}}\frac{|\nabla
u|^{2}}{|x|^{2}\left(  1-\left(  \frac{R}{|x|}\right)  ^{-N}\right)  ^{2}}dx.
\]

\item[c)] (Hardy-Rellich inequality of Brezis-V\'{a}zquez type) For all $u\in
C_{0}^{\infty}\left(  B_{R}\right)  :$%
\[
\int_{B_{R}}|\Delta u|^{2}dx\geq\frac{N^{2}}{4}\int_{B_{R}}\frac{|\nabla
u|^{2}}{|x|^{2}}dx+\frac{z_{0}^{2}}{R^{2}}\int_{B_{R}}|\nabla u|^{2}dx.
\]

\end{itemize}
\end{corollary}

\begin{proof}
$\left(  1,\dfrac{N^{2}}{4}\dfrac{1}{r^{2}}\right)  $ is a $\left(
N+2\right)  $-dimensional Bessel pairs on $\left(  0,\infty\right)  $,
$\left(  1,\frac{N^{2}}{4}\frac{1}{r^{2}\left(  1-\left(  \frac{R}{r}\right)
^{-N}\right)  ^{2}}\right)  $ is a $\left(  N+2\right)  $-dimensional Bessel
pairs on $\left(  0,R\right)  $, and $\left(  1,\frac{N^{2}}{4}\frac{1}{r^{2}%
}+\frac{z_{0}^{2}}{R^{2}}\right)  $ is a $\left(  N+2\right)  $-dimensional
Bessel pairs on $\left(  0,R\right)  .$
\end{proof}

Since we can obtain the Hardy-Rellich type inequalities (\ref{wwHR}) without
the extra condition that $W_{2}-\frac{2}{r^{2}}\geq0$ on $\left(  0,R\right)
$ as in \cite{GM}, we can derive several new Hardy-Rellich type inequalities
that can not be attained by Theorem \ref{t1}, especially when the potential
$W_{2}$ is not always nonnegative. For instance, we obtain the following
Hardy-Rellich type inequalities

\begin{corollary}
\label{C2}For all $u\in C_{0}^{\infty}(\mathbb{R}^{N})$, $N\geq5$, we have

\begin{itemize}
\item[a)]
\[
\int_{\mathbb{R}^{N}}|\Delta u|^{2}dx\geq\int_{\mathbb{R}^{N}}\left(
N+2-|x|^{2}\right)  |\nabla u|^{2}dx.
\]

\item[b)]
\[
\int_{\mathbb{R}^{N}}|\Delta u|^{2}dx\geq\int_{\mathbb{R}^{N}}\left(
\dfrac{N+1}{|x|}-1\right)  |\nabla u|^{2}dx.
\]

\item[c)] for $b<1:$%
\[
\int_{\mathbb{R}^{N}}|\Delta u|^{2}dx\geq\int_{\mathbb{R}^{N}}\left(
\frac{N+1-b}{|x|^{b+1}}-\frac{1}{|x|^{2b}}\right)  |\nabla u|^{2}dx.
\]

\item[d)] for $b>1:$%
\[
\int_{\mathbb{R}^{N}}|\Delta u|^{2}dx\geq\int_{\mathbb{R}^{N}}\left(
\frac{N+b-1}{|x|^{b+1}}-\frac{1}{|x|^{2b}}\right)  |\nabla u|^{2}dx.
\]

\end{itemize}
\end{corollary}

\begin{proof}
We apply Theorem \ref{T1} to the following $\left(  N+2\right)  $-dimensional
Bessel pairs on $\left(  0,\infty\right)  $: $\left(  1,N+2-r^{2}\right)
$,\ $\left(  1,\dfrac{N+1}{r}-1\right)  $, $\left(  1,\dfrac{N+1-b}{r^{b+1}%
}-\frac{1}{r^{2b}}\right)  $ for $b<1$, and $\left(  1,\dfrac{N+b-1}{r^{b+1}%
}-\frac{1}{r^{2b}}\right)  $ for $b>1$, respectively.
\end{proof}

Obviously, using the standard scaling argument, we can also deduce the
following Caffarelli-Kohn-Nirenberg inequalities from Corollary \ref{C2}:

\begin{corollary}
For all $u\in C_{0}^{\infty}(\mathbb{R}^{N})$, $N\geq5$, we have

\begin{itemize}
\item[a)] {(Second order Heisenberg Uncertainty Principle)}
\[
\left(  \int_{\mathbb{R}^{N}}|\Delta u|^{2}dx\right)  ^{\frac{1}{2}}\left(
\int_{\mathbb{R}^{N}}|x|^{2}|\nabla u|^{2}dx\right)  ^{\frac{1}{2}}\geq
\frac{N+2}{2}\int_{\mathbb{R}^{N}}|\nabla u|^{2}dx.
\]

\item[b)] {(Second order Hydrogen Uncertainty Principle)}
\[
\left(  \int_{\mathbb{R}^{N}}|\Delta u|^{2}dx\right)  ^{\frac{1}{2}}\left(
\int_{\mathbb{R}^{N}}|\nabla u|^{2}dx\right)  ^{\frac{1}{2}}\geq\frac{N+1}%
{2}\int_{\mathbb{R}^{N}}\frac{|\nabla u|^{2}}{|x|}dx.
\]

\item[c)] for $b<1:$%
\[
\left(  \int_{\mathbb{R}^{N}}|\Delta u|^{2}dx\right)  ^{\frac{1}{2}}\left(
\int_{\mathbb{R}^{N}}\frac{|\nabla u|^{2}}{|x|^{2b}}dx\right)  ^{\frac{1}{2}%
}\geq\frac{N+1-b}{2}\int_{\mathbb{R}^{N}}\frac{|\nabla u|^{2}}{|x|^{b+1}}dx.
\]

\item[d)] for $b>1:$%
\[
\left(  \int_{\mathbb{R}^{N}}|\Delta u|^{2}dx\right)  ^{\frac{1}{2}}\left(
\int_{\mathbb{R}^{N}}\frac{|\nabla u|^{2}}{|x|^{2b}}dx\right)  ^{\frac{1}{2}%
}\geq\frac{N+b-1}{2}\int_{\mathbb{R}^{N}}\frac{|\nabla u|^{2}}{|x|^{b+1}}dx.
\]

\end{itemize}
\end{corollary}

See, e.g., \cite{CKN, CC09, CFL23, CFL22, CFL21, DN21, Fly20} for related results.

It is worth mentioning that symmetry breaking is never an issue for the Hardy
type inequalities of the form (\ref{wH}). Indeed, in \cite{DLL22}, Duy, Lam
and Lu have showed that (\ref{wH}) for all $u\in C_{0}^{\infty}\left(
B_{R}\setminus\left\{  0\right\}  \right)  $ is equivalent to (\ref{wH}) for
all radial functions $u\in C_{0}^{\infty}\left(  B_{R}\setminus\left\{
0\right\}  \right)  $. In the language of spherical harmonics decomposition,
that means that the Hardy type inequalities (\ref{wH}) holds for all $u\in
C_{0}^{\infty}\left(  B_{R}\setminus\left\{  0\right\}  \right)  $ if and only
if it holds at the node $k=0$ for all $u\in C_{0}^{\infty}\left(
B_{R}\setminus\left\{  0\right\}  \right)  $.

Before we continue, we recall that as mentioned earlier, a commonly used
method to prove $L^{p}$-Hardy type inequalities is to use the Schwarz
rearrangement. The argument is as follows: by the well-known
P\'{o}lya-Szeg\"{o} inequality and Hardy-Littlewood inequality, we have
\[
\left\vert \frac{N-p}{p}\right\vert ^{p}%
{\displaystyle\int\limits_{\mathbb{R}^{N}}}
\frac{\left\vert u\right\vert ^{p}}{\left\vert x\right\vert ^{p}}%
dx\leq\left\vert \frac{N-p}{p}\right\vert ^{p}%
{\displaystyle\int\limits_{\mathbb{R}^{N}}}
\frac{\left\vert u^{\ast}\right\vert ^{p}}{\left\vert x\right\vert ^{p}}dx
\]
and
\[%
{\displaystyle\int\limits_{\mathbb{R}^{N}}}
\left\vert \nabla u\right\vert ^{p}dx\geq%
{\displaystyle\int\limits_{\mathbb{R}^{N}}}
\left\vert \nabla u^{\ast}\right\vert ^{p}dx.
\]
Here $u^{\ast}$ is the radial nonincreasing rearrangement of $u$. As a
consequence, it is now enough to assume that $u$ is radial. Therefore we can
convert the problem to the one-dimensional case in which we can apply the
one-dimensional Hardy inequality. More generally, in \cite{DLL22}, the authors
established the following general symmetrization principle for the Hardy type
inequalities with nonradial weights of the form $A(\left\vert x\right\vert
)x^{P}$:

\begin{theorem}
[Duy-Lam-Lu \cite{DLL22}]\label{T1a}Let $p>1,$ $0<R$ $\leq$ $\infty,$ $A$ and
$B$ be functions on $(0,R)$ with $A\geq0$. Then the following are equivalent:
\end{theorem}

$\left(  A\right)  $ $%
{\displaystyle\int\limits_{B_{R}^{\ast}}}
A(\left\vert x\right\vert )\left\vert \nabla u\right\vert ^{p}x^{P}dx\geq%
{\displaystyle\int\limits_{B_{R}^{\ast}}}
B(\left\vert x\right\vert )\left\vert u\right\vert ^{p}x^{P}dx$ for all $u$
$\in C_{0}^{\infty}(B_{R}^{\ast})$

$\left(  B\right)  $ $%
{\displaystyle\int\limits_{B_{R}^{\ast}}}
A(\left\vert x\right\vert )\left\vert \mathcal{R}u\right\vert ^{p}x^{P}dx\geq%
{\displaystyle\int\limits_{B_{R}^{\ast}}}
B(\left\vert x\right\vert )\left\vert u\right\vert ^{p}x^{P}dx$ for all $u$
$\in C_{0}^{\infty}(B_{R}^{\ast})$

$\left(  C\right)  $ $%
{\displaystyle\int\limits_{B_{R}^{\ast}}}
A(\left\vert x\right\vert )\left\vert \nabla u\right\vert ^{p}x^{P}dx\geq%
{\displaystyle\int\limits_{B_{R}^{\ast}}}
B(\left\vert x\right\vert )\left\vert u\right\vert ^{p}x^{P}dx$ for all radial
functions $u$ $\in C_{0}^{\infty}(B_{R}^{\ast}).$%

\medskip

Here $x^{P}=\left\vert x_{1}\right\vert ^{P_{1}}...\left\vert x_{N}\right\vert
^{P_{N}}$, $P_{1}\geq0,...,$ $P_{N}\geq0$, is the monomial weight,
$\mathbb{R}_{\ast}^{N}=\left\{  \left(  x_{1},...,x_{N}\right)  \in
\mathbb{R}^{N}:x_{i}>0\text{ whenever }P_{i}>0\right\}  $, $B_{R}^{\ast}%
=B_{R}\cap\mathbb{R}_{\ast}^{N}$ and $\mathcal{R}:=\frac{x}{\left\vert
x\right\vert }\cdot\nabla$.

\medskip

Nevertheless, for the second and higher order Hardy-Rellich or Rellich type
inequalities, the P\'{o}lya-Szeg\"{o} inequality no longer holds for higher
order derivatives. Therefore, it becomes very interesting to know whether a
Hardy-Rellich type inequality that holds for all radial functions will
continue to hold for all nonradial functions as well. More precisely, it is
desirable to know under what conditions on the weights $V_{2}$ and $W_{2}$ so
that the Hardy-Rellich type inequality \eqref{weights} holds for all functions
$u$ if and only if the same inequality holds for all radial functions and
under what conditions on weights $V_{1}$ and $W_{1}$ so that the Hardy-Rellich
inequality \eqref{weights2} holds for all functions if and only if the same
inequality holds for all radial functions. In fact, Ghoussoub and Moradifam
raised this an open question (see \cite[page 91, Open problem 6]{GM1}).

\medskip

Our next goal is to investigate this phenomenon for the Hardy-Rellich type
inequalities (\eqref{wHR}).

\begin{theorem}
\label{T3}\textit{Let }$0<R\leq\infty$\textit{, }$V_{2}\geq0$ and $W_{2}%
$\textit{ be smooth functions on }$\left(  0,R\right)  $ satisfying%
\begin{align}
&  2\int_{0}^{R}V_{2}(r)r^{N-3}\left\vert u^{\prime}\right\vert ^{2}%
dr-\int_{0}^{R}V_{2}^{\prime\prime}\left(  r\right)  r^{N-3}\left\vert
u\right\vert ^{2}dr-\left(  N-5\right)  \int_{0}^{R}V_{2}^{\prime}\left(
r\right)  r^{N-4}\left\vert u\right\vert ^{2}dr\label{ConM2}\\
&  \hspace{0.2in} +(5N-9)\int_{0}^{R}V_{2}(r)r^{N-5}\left\vert u\right\vert
^{2}dr-\int_{0}^{R}W_{2}(r)r^{N-3}\left\vert u\right\vert ^{2}dr\geq0\text{
}\forall u\in C_{0}^{\infty}\left(  0,R\right)  .\nonumber
\end{align}
We have
\[
\int_{B_{R}}V_{2}(|x|)|\Delta v|^{2}dx\geq\int_{B_{R}}W_{2}(|x|)|\nabla
v|^{2}dx\;\text{for all}\;v\in C_{0}^{\infty}\left(  B_{R}\setminus\left\{
0\right\}  \right)
\]
if and only if
\begin{align}
\label{condition1} &  \int_{0}^{R}V_{2}(r)r^{N-1}\left\vert u^{\prime\prime
}\right\vert ^{2}dr+(N-1)\int_{0}^{R}V_{2}(r)r^{N-3}\left\vert u^{\prime
}\right\vert ^{2}dr-(N-1)\int_{0}^{R}V_{2}^{\prime}\left(  r\right)
r^{N-2}\left\vert u^{\prime}\right\vert ^{2}dr\\
&  \geq\int_{0}^{R}W_{2}(r)r^{N-1}\left\vert u^{\prime}\right\vert ^{2}dr,
\tag{node $k=0$}%
\end{align}
and%
\begin{align}
\label{condition2} &  \int_{0}^{R}V_{2}(r)r^{N-1}\left\vert u^{\prime\prime
}\right\vert ^{2}dr+((N-1)+2c_{1})\int_{0}^{R}V_{2}(r)r^{N-3}\left\vert
u^{\prime}\right\vert ^{2}dr\\
&  \hspace{0.2in}+(c_{1}^{2}+2(N-4)c_{1})\int_{0}^{R}V_{2}(r)r^{N-5}\left\vert
u\right\vert ^{2}dr-(N-1)\int_{0}^{R}V_{2}^{\prime}\left(  r\right)
r^{N-2}\left\vert u^{\prime}\right\vert ^{2}dr\nonumber\\
&  \hspace{0.2in}-(N-5)c_{1}\int_{0}^{R}V_{2}^{\prime}\left(  r\right)
r^{N-4}\left\vert u\right\vert ^{2}dr-c_{1}\int_{0}^{R}V_{2}^{\prime\prime
}\left(  r\right)  r^{N-3}\left\vert u\right\vert ^{2}dr\nonumber\\
&  \geq\int_{0}^{R}W_{2}(r)r^{N-1}\left\vert u^{\prime}\right\vert
^{2}dr+c_{1}\int_{0}^{R}W_{2}(r)r^{N-3}\left\vert u\right\vert ^{2}dr,
\tag{node $k=1$}%
\end{align}
for all $u\in C_{0}^{\infty}((0,R))$.
\end{theorem}

Similarly as above, in several situations, we can replace the condition
(\ref{ConM2}) by a stricter, but more natural, condition%

\begin{equation}
V_{2}^{\prime\prime}(r)-3\dfrac{V_{2}^{\prime}(r)}{r}-\left(  3N-5\right)
\dfrac{V_{2}(r)}{r^{2}}\leq0\text{ on }\left(  0,R\right)  . \label{Con3}%
\end{equation}

\begin{theorem}
\label{T3a}\textit{Let }$0<R\leq\infty$\textit{, }$V_{2}\geq0$ and $W_{2}%
$\textit{ be smooth functions on }$\left(  0,R\right)  $ satisfying
(\ref{ConM2}) or (\ref{Con3}). Then
\[
\int_{B_{R}}V_{2}(|x|)|\Delta v|^{2}dx\geq\int_{B_{R}}W_{2}(|x|)|\nabla
v|^{2}dx\;\text{for all}\;v\in C_{0}^{\infty}\left(  B_{R}\setminus\left\{
0\right\}  \right)
\]
holds if and only if \eqref{condition1} and \eqref{condition2} hold for all
$u\in C_{0}^{\infty}((0,R))$.
\end{theorem}

The meaning of Theorem \ref{T3} is that the main contributors to the
optimality of the weighted Hardy-Rellich type inequalities are, in terms of
spherical harmonics decomposition, at the nodes $k=0$ and $k=1$. Furthermore,
we will show that with just a little stricter condition (\ref{ConM}) (note
that condition (\ref{ConM}) implies (\ref{ConM2}); similarly (\ref{Con2})
implies (\ref{Con3})), the node $k=1$ in Theorem \ref{T3} holds automatically.
Therefore, we get the following result that can be considered as a
symmetrization type result for the Hardy-Rellich type inequalities:

\begin{theorem}
\label{T3.1} \textit{Let }$N\geq1$,\textit{ }$0<R\leq\infty$\textit{, }%
$V_{2}\geq0$ and $W_{2}$\textit{ be smooth functions on }$\left(  0,R\right)
$ satisfying (\ref{ConM}) or (\ref{Con2}). We have
\[
\int_{B_{R}}V_{2}(|x|)|\Delta u|^{2}dx\geq\int_{B_{R}}W_{2}(|x|)|\nabla
u|^{2}dx\;\text{for all}\;u\in C_{0}^{\infty}\left(  B_{R}\setminus\left\{
0\right\}  \right)
\]
if and only if
\[
\int_{B_{R}}V_{2}(|x|)|\Delta u|^{2}dx\geq\int_{B_{R}}W_{2}(|x|)|\nabla
u|^{2}dx\;\text{for all radial functions}\;u\in C_{0}^{\infty}\left(
B_{R}\setminus\left\{  0\right\}  \right)  .
\]

\end{theorem}

Theorem \ref{T3.1} is interesting in the sense that a rearrangement argument
is missing when dealing with higher order operator. Indeed, it is well-known
that a P\'{o}lya--Szeg\H{o} type inequality of the form%
\[
\int_{\mathbb{R}^{N}}V(|x|)\left\vert \Delta u\right\vert ^{2}dx\geq
\int_{\mathbb{R}^{N}}V(|x|)\left\vert \Delta u^{\ast}\right\vert ^{2}dx
\]
does not hold in general. More importantly, the weight $W_{2}$ is not required
to be nonnegative or radially decreasing. Therefore, for instance, one cannot
apply the classical Hardy-Littlewood inequality.

Theorem \ref{T3.1} provides some sufficient conditions to prevent the symmetry
breaking in the Hardy-Rellich type inequalities. Moreover, if these conditions
(\ref{ConM}) and (\ref{Con2}) fail, then the symmetry breaking phenomenon may
happen in the Hardy-Rellich type inequalities. Indeed, when $V_{2}=1$ and
$W=c\frac{1}{|x|^{2}}$, then (\ref{ConM}) and (\ref{Con2}) fail if and only if
$N\leq4$. In this situation, it was showed in \cite{Bec08} that the symmetry
breaking in the Hardy-Rellich inequality $\int_{\mathbb{R}^{N}}\left\vert
\Delta u\right\vert ^{2}dx\geq c\int_{\mathbb{R}^{N}}\frac{|\nabla u|^{2}%
}{|x|^{2}}dx$ happens. That is, the optimality of this Hardy-Rellich
inequality happens at the node $k=1$.

In the special case $V_{2}=1$, we obtain

\begin{theorem}
\label{T4}\textit{Let }$N\geq1$,\textit{ }$0<R\leq\infty$\textit{, }$W_{2}%
$\textit{ be a smooth function on }$\left(  0,R\right)  $. We have that
\[
\int_{B_{R}}|\Delta u|^{2}dx\geq\int_{B_{R}}W_{2}(|x|)|\nabla u|^{2}%
dx\;\text{for all}\;u\in C_{0}^{\infty}(B_{R}\setminus\left\{  0\right\}  )
\]
if and only if
\begin{equation}
\int_{0}^{R}r^{N-1}\left\vert u^{\prime\prime}\right\vert ^{2}dr+(N-1)\int
_{0}^{R}r^{N-3}\left\vert u^{\prime}\right\vert ^{2}dr\geq\int_{0}^{R}%
W_{2}(r)r^{N-1}\left\vert u^{\prime}\right\vert ^{2}dr \tag{node $k=0$}%
\end{equation}
and
\begin{align}
&  \int_{0}^{R}r^{N-1}\left\vert u^{\prime\prime}\right\vert ^{2}%
dr+(N-1+2c_{1})\int_{0}^{R}r^{N-3}\left\vert u^{\prime}\right\vert
^{2}dr+(c_{1}^{2}+2(N-4)c_{1})\int_{0}^{R}r^{N-5}|u|^{2}dr\nonumber\\
&  \geq\int_{0}^{R}W_{2}(r)r^{N-1}\left\vert u^{\prime}\right\vert
^{2}dr+c_{1}\int_{0}^{R}W_{2}(r)r^{N-3}|u|^{2}dr, \tag{node $k=1$}%
\end{align}
for all $u\in C_{0}^{\infty}((0,R))$.
\end{theorem}

Moreover, as a consequence of Theorem \ref{T3.1}, we have the following result
addressing the open problem raised by Ghoussoub and Marafadim \cite{GM1}.

\begin{theorem}
[Symmetrization Principle]\label{T5}\textit{Let }$N\geq5$,\textit{ }%
$0<R\leq\infty$\textit{, }$W_{2}$\textit{ be a smooth function on }$\left(
0,R\right)  $. We have that
\begin{equation}
\int_{B_{R}}|\Delta u|^{2}dx\geq\int_{B_{R}}W_{2}(|x|)|\nabla u|^{2}%
dx\;\text{for all}\;u\in C_{0}^{\infty}(B_{R}\setminus\left\{  0\right\}  )
\label{HR2}%
\end{equation}
if and only if
\[
\int_{B_{R}}|\Delta u|^{2}dx\geq\int_{B_{R}}W_{2}(|x|)|\nabla u|^{2}%
dx\;\text{for all radial functions}\;u\in C_{0}^{\infty}(B_{R}\setminus
\left\{  0\right\}  ).
\]

\end{theorem}

\begin{remark}
Our Theorem \ref{T5} asserts that when the dimension $N\geq5$, the symmetry
breaking phenomenon cannot happen in the Hardy-Rellich type inequalities
(\ref{HR2}), and therefore partially answers an open question raised by
Ghoussoub and Moradifam in \cite[page 91, Open problem 6]{GM1}. Moreover, as
mentioned earlier, the dimensional condition $N\geq5$ is optimal in the sense
that the symmetry breaking may happen in the Hardy-Rellich type inequalities
when $N\leq4$.
\end{remark}

As a consequence of Theorem \ref{T5}, Theorem \ref{T1}, Theorem \ref{T1a}, and
Theorem 4.1.1 in \cite{GM1}, we obtain the following result about the
necessary and sufficient conditions for the Hardy-Rellich type inequalities
that appears to be new in the literature:

\begin{theorem}
\label{T5.1}\textit{Let }$N\geq5$,\textit{ }$0<R\leq\infty$\textit{, }$W_{2}%
$\textit{ be a smooth function on }$\left(  0,R\right)  $. Then the following
statements are equivalent:

\begin{itemize}
\item[1.] There exists $c>0$ such that for all$\;u\in C_{0}^{\infty}(B_{R}%
\cap\mathbb{R}^{N}\setminus\left\{  0\right\}  ):$
\[
\int_{B_{R}\cap\mathbb{R}^{N}}|\Delta u|^{2}dx\geq c\int_{B_{R}\cap
\mathbb{R}^{N}}W_{2}(|x|)|\nabla u|^{2}dx.
\]

\item[2.] There exists $c>0$ such that for all radial functions$\;u\in
C_{0}^{\infty}(B_{R}\cap\mathbb{R}^{N}\setminus\left\{  0\right\}  ):$
\[
\int_{B_{R}\cap\mathbb{R}^{N}}|\Delta u|^{2}dx\geq c\int_{B_{R}\cap
\mathbb{R}^{N}}W_{2}(|x|)|\nabla u|^{2}dx.
\]

\item[3.] There exists $c>0$ such that for all$\;u\in C_{0}^{\infty}(B_{R}%
\cap\mathbb{R}^{N+2}\setminus\left\{  0\right\}  ):$
\[
\int_{B_{R}\cap\mathbb{R}^{N+2}}|\nabla u|^{2}dx\geq c\int_{B_{R}%
\cap\mathbb{R}^{N+2}}W_{2}(|x|)|u|^{2}dx.
\]

\item[4.] There exists $c>0$ such that for all radial functions$\;u\in
C_{0}^{\infty}(B_{R}\cap\mathbb{R}^{N+2}\setminus\left\{  0\right\}  ):$
\[
\int_{B_{R}\cap\mathbb{R}^{N+2}}|\nabla u|^{2}dx\geq c\int_{B_{R}%
\cap\mathbb{R}^{N+2}}W_{2}(|x|)|u|^{2}dx.
\]

\item[5.] $\left(  1,cW_{2}\right)  $ is a $\left(  N+2\right)  $-dimensional
Bessel pair on $\left(  0,R\right)  $ for some $c>0$.
\end{itemize}
\end{theorem}

In the same spirit, we also study the following symmetrization result for the
Rellich type inequalities:

\begin{theorem}
[Symmetrization Principle]\label{T6}\textit{Let }$N\geq3$,\textit{ }%
$0<R\leq\infty$\textit{, }$W$\textit{ be a smooth function on }$\left(
0,R\right)  $. We have that
\[
\int_{B_{R}}|\Delta u|^{2}dx\geq\int_{B_{R}}W(|x|)|u|^{2}dx\;\text{for
all}\;u\in C_{0}^{\infty}(B_{R}\setminus\left\{  0\right\}  )
\]
if and only if
\[
\int_{B_{R}}|\Delta u|^{2}dx\geq\int_{B_{R}}W(|x|)|u|^{2}dx\;\text{for all
radial functions}\;u\in C_{0}^{\infty}(B_{R}\setminus\left\{  0\right\}  ).
\]

\end{theorem}

This paper is organized as follows: In section 2, we will provide some
backgrounds about the spherical harmonics decompositions, and will also
present some useful computations. Our main results will be proved in section 3.

\section{Spherical Harmonics Decomposition}

First, we recall some backgrounds about the spherical harmonics decomposition.
Let $u\in C_{0}^{\infty}(B_{R})$, $0<R\leq\infty$. We can decompose $u$ into
spherical harmonics as follows:
\[
u(x)=u(r\sigma)=\sum_{k=0}^{\infty}u_{k}(r)\phi_{k}(\sigma),
\]
where $\phi_{k}(\sigma)$ are the orthonormal eigenfunctions of the
Laplace-Beltrami operator with corresponding eigenvalues $c_{k}=k(N+k-2)$,
$k\geq0$. Moreover, the components $u_{k}\in C_{0}^{\infty}[0,R)$ and satisfy
$u_{k}(r)=O(r^{k})$, $u_{k}^{\prime}(r)=O(r^{k-1})$ as $r\downarrow0$. In
particular,
\[
\phi_{0}(\sigma)=1,\;c_{0}=0\;\text{and}\;u_{0}(r)=\dfrac{1}{|\partial B_{r}%
|}\int_{\partial B_{r}}uds.
\]

Now, we consider%
\begin{align*}
\int_{B_{R}}|\Delta u|^{2}dx  &  =\int_{S^{N-1}}\int_{0}^{R}\left\vert
\Delta_{r}u+\dfrac{1}{r^{2}}\Delta_{S^{N-1}}u\right\vert ^{2}r^{N-1}%
drd\sigma\\
&  =\sum_{k=0}^{\infty}\left(  \int_{0}^{R}r^{N-1}\left\vert \Delta_{r}%
u_{k}\right\vert ^{2}dr+c_{k}^{2}\int_{0}^{R}r^{N-5}\left\vert u_{k}%
\right\vert ^{2}dr-2c_{k}\int_{0}^{R}u_{k}\Delta_{r}u_{k}r^{N-3}dr\right)  ,
\end{align*}
where
\[
\Delta_{r}:=\partial_{rr}+\dfrac{N-1}{r}\partial_{r}.
\]
Moreover,
\begin{align*}
\int_{0}^{R}r^{N-1}\left\vert \Delta_{r}u_{k}\right\vert ^{2}dr  &  =\int
_{0}^{R}r^{N-1}\left\vert u_{k}^{\prime\prime}+\dfrac{N-1}{r}u_{k}^{\prime
}\right\vert ^{2}dr\\
&  =\int_{0}^{R}r^{N-1}\left\vert u_{k}^{\prime\prime}\right\vert
^{2}dr+2(N-1)\int_{0}^{R}r^{N-2}u_{k}^{\prime}u_{k}^{\prime\prime}dr\\
&  \hspace{0.2in}+(N-1)^{2}\int_{0}^{R}r^{N-3}\left\vert u_{k}^{\prime
}\right\vert ^{2}dr\\
&  =\int_{0}^{R}r^{N-1}\left\vert u_{k}^{\prime\prime}\right\vert
^{2}dr+(N-1)\int_{0}^{R}r^{N-3}\left\vert u_{k}^{\prime}\right\vert ^{2}dr.
\end{align*}
Furthermore,
\begin{align*}
2c_{k}\int_{0}^{R}u_{k}\Delta_{r}u_{k}r^{N-3}dr  &  =2c_{k}\int_{0}^{R}%
r^{N-3}u_{k}u_{k}^{\prime\prime}dr+2(N-1)c_{k}\int_{0}^{R}r^{N-4}u_{k}%
u_{k}^{\prime}dr\\
&  =-2c_{k}\int_{0}^{R}r^{N-3}\left\vert u_{k}^{\prime}\right\vert
^{2}dr-2(N-4)c_{k}\int_{0}^{R}r^{N-5}|u_{k}|^{2}dr.
\end{align*}
Therefore,
\begin{align*}
\int_{B_{R}}|\Delta u|^{2}dx  &  =\sum_{k=0}^{\infty}\left(  \int_{0}%
^{R}r^{N-1}\left\vert u_{k}^{\prime\prime}\right\vert ^{2}dr+((N-1)+2c_{k}%
)\int_{0}^{R}r^{N-3}\left\vert u_{k}^{\prime}\right\vert ^{2}dr\right. \\
&  \hspace{0.2in}\left.  +(c_{k}^{2}+2(N-4)c_{k})\int_{0}^{R}r^{N-5}%
|u_{k}|^{2}dr\right)  .
\end{align*}

Similarly,
\begin{align*}
\int_{B_{R}}W(|x|)|\nabla u|^{2}dx  &  =\int_{S^{N-1}}\int_{0}^{R}\left(
W(r)|\partial_{r}u|^{2}+W(r)\dfrac{|\nabla_{S^{N-1}}u|^{2}}{r^{2}}\right)
r^{N-1}drd\sigma\\
&  =\sum_{k=0}^{\infty}\left(  \int_{0}^{R}W(r)r^{N-1}|u_{k}^{\prime}%
|^{2}dr+c_{k}\int_{0}^{R}W(r)r^{N-3}|u_{k}|^{2}dr\right)  .
\end{align*}

We also have%
\begin{align*}
\int_{B_{R}}V(|x|)|\Delta u|^{2}dx  &  =\sum_{k=0}^{\infty}\left[  \int
_{0}^{R}V(r)r^{N-1}\left\vert u_{k}^{\prime\prime}\right\vert ^{2}%
dr+((N-1)+2c_{k})\int_{0}^{R}V(r)r^{N-3}\left\vert u_{k}^{\prime}\right\vert
^{2}dr\right. \\
&  \hspace{0.2in}+(c_{k}^{2}+2(N-4)c_{k})\int_{0}^{R}V(r)r^{N-5}\left\vert
u_{k}\right\vert ^{2}dr-(N-1)\int_{0}^{R}V^{\prime}\left(  r\right)
r^{N-2}\left\vert u_{k}^{\prime}\right\vert ^{2}dr\\
&  \hspace{0.2in}\left.  -(N-5)c_{k}\int_{0}^{R}V^{\prime}\left(  r\right)
r^{N-4}\left\vert u_{k}\right\vert ^{2}dr-c_{k}\int_{0}^{R}V^{\prime\prime
}\left(  r\right)  r^{N-3}\left\vert u_{k}\right\vert ^{2}dr\right]
\end{align*}
and
\[
\int_{B_{R}}V(|x|)|u|^{2}dx=\sum_{k=0}^{\infty}\int_{0}^{R}V(r)r^{N-1}%
|u_{k}|^{2}dr.
\]

\section{Proofs of main results}

Though part 1 of Theorem \ref{T1} is is just a special case of Theorem
\ref{T2}, we still present here a proof for the part 1 of Theorem \ref{T1} in
order to illustrate our main ideas.

\begin{proof}
[Proof of Theorem \ref{T1}, part 1]Recall that by using spherical harmonics
decomposition $u(x)=\sum_{k=0}^{\infty}u_{k}(r)\phi_{k}(\sigma)$, we get from
the previous section that
\begin{align*}
\int_{B_{R}}|\Delta u|^{2}dx  &  =\sum_{k=0}^{\infty}\left(  \int_{0}%
^{R}r^{N-1}|u_{k}^{\prime\prime}|^{2}dr+((N-1)+2c_{k})\int_{0}^{R}%
r^{N-3}|u_{k}^{\prime}|^{2}dr\right. \\
&  \hspace{0.2in}\left.  +(c_{k}^{2}+2(N-4)c_{k})\int_{0}^{R}r^{N-5}%
|u_{k}|^{2}dr\right)
\end{align*}
and
\[
\int_{B_{R}}W_{2}(|x|)|\nabla u|^{2}dx=\sum_{k=0}^{\infty}\left(  \int_{0}%
^{R}W_{2}(r)r^{N-1}|u_{k}^{\prime}|^{2}dr+c_{k}\int_{0}^{R}W_{2}%
(r)r^{N-3}|u_{k}|^{2}dr\right)  .
\]
Therefore, it suffices to prove that
\begin{align*}
\int_{0}^{R}r^{N-1}|u_{k}^{\prime\prime}|^{2}dr  &  +((N-1)+2c_{k})\int
_{0}^{R}r^{N-3}|u_{k}^{\prime}|^{2}dr+(c_{k}^{2}+2(N-4)c_{k})\int_{0}%
^{R}r^{N-5}|u_{k}|^{2}dr\\
&  -\int_{0}^{R}W_{2}(r)r^{N-1}|u_{k}^{\prime}|^{2}dr-c_{k}\int_{0}^{R}%
W_{2}(r)r^{N-3}|u_{k}|^{2}dr\geq0.
\end{align*}
Firstly, we consider for the non $c_{k}$ term and will prove that
\[
\int_{0}^{R}r^{N-1}|u_{k}^{\prime\prime}|^{2}dr+(N-1)\int_{0}^{R}r^{N-3}%
|u_{k}^{\prime}|^{2}dr-\int_{0}^{R}W_{2}(r)r^{N-1}|u_{k}^{\prime}|^{2}%
dr\geq0.
\]
Indeed, by using the fact that $(1,W_{2})$ is a $\left(  N+2\right)
$-dimensional Bessel pair, we have
\[
\int_{0}^{R}r^{N+1}\left\vert \left(  \dfrac{u_{k}^{\prime}}{r}\right)
^{\prime}\right\vert ^{2}dr\geq\int_{0}^{R}W_{2}(r)r^{N+1}\left\vert
\dfrac{u_{k}^{\prime}}{r}\right\vert ^{2}dr,
\]
which is equivalent to
\[
\int_{0}^{R}r^{N-1}|u_{k}^{\prime\prime}|^{2}dr-2\int_{0}^{R}r^{N-2}%
u_{k}^{\prime\prime}u_{k}^{\prime}dr+\int_{0}^{R}r^{N-3}|u_{k}^{\prime}%
|^{2}dr\geq\int_{0}^{R}W_{2}(r)r^{N-1}|u_{k}^{\prime}|^{2}dr.
\]
Moreover,
\[
-2\int_{0}^{R}r^{N-2}u_{k}^{\prime\prime}u_{k}^{\prime}dr=(N-2)\int_{0}%
^{R}r^{N-3}|u_{k}^{\prime}|^{2}dr.
\]
Then, we obtain the desired inequality for the non $c_{k}$ term. Next, we will
deal with the $c_{k}$ terms, i.e.
\[
2c_{k}\int_{0}^{R}r^{N-3}|u_{k}^{\prime}|^{2}dr+(c_{k}^{2}+2(N-4)c_{k}%
)\int_{0}^{R}r^{N-5}|u_{k}|^{2}dr-c_{k}\int_{0}^{R}W_{2}(r)r^{N-3}|u_{k}%
|^{2}dr\geq0.
\]
Since $c_{k}\geq N-1$ when $k\geq1$, it is enough to prove that
\[
2\int_{0}^{R}r^{N-3}|u_{k}^{\prime}|^{2}dr+3(N-3)\int_{0}^{R}r^{N-5}%
|u_{k}|^{2}dr-\int_{0}^{R}W_{2}(r)r^{N-3}|u_{k}|^{2}dr\geq0.
\]
Indeed, setting $u_{k}=r^{2}v_{k}$, we have
\begin{align*}
\int_{0}^{R}r^{N-3}|u_{k}^{\prime}|^{2}dr  &  =4\int_{0}^{R}r^{N-1}|v_{k}%
|^{2}dr+\int_{0}^{R}r^{N+1}|v_{k}^{\prime}|^{2}dr+4\int_{0}^{R}r^{N}v_{k}%
v_{k}^{\prime}dr.\\
&  =(4-2N)\int_{0}^{R}r^{N-1}|v_{k}|^{2}dr+\int_{0}^{R}r^{N+1}|v_{k}^{\prime
}|^{2}dr\\
&  =(4-2N)\int_{0}^{R}r^{N-5}|u_{k}|^{2}dr+\int_{0}^{R}r^{N+1}|v_{k}^{\prime
}|^{2}dr.
\end{align*}
Thus,
\begin{align*}
\int_{0}^{R}r^{N-3}|u_{k}^{\prime}|^{2}dr+(2N-4)\int_{0}^{R}r^{N-5}|u_{k}%
|^{2}dr=\int_{0}^{R}r^{N+1}|v_{k}^{\prime}|^{2}dr  &  \geq\int_{0}^{R}%
W_{2}(r)r^{N+1}|v_{k}|^{2}dr\\
&  \geq\int_{0}^{R}W_{2}(r)r^{N-3}|u_{k}|^{2}dr.
\end{align*}
Hence,
\begin{align*}
&  2\int_{0}^{R}r^{N-3}|u_{k}^{\prime}|^{2}dr+3(N-3)\int_{0}^{R}r^{N-5}%
|u_{k}|^{2}dr\\
&  =\left(  \int_{0}^{R}r^{N-3}|u_{k}^{\prime}|^{2}dr+(N-5)\int_{0}^{R}%
r^{N-5}|u_{k}|^{2}dr\right) \\
&  \hspace{0.2in} +\left(  \int_{0}^{R}r^{N-3}|u_{k}^{\prime}|^{2}%
dr+(2N-4)\int_{0}^{R}r^{N-5}|u_{k}|^{2}dr\right) \\
&  \geq\int_{0}^{R}W_{2}(r)r^{N-3}|u_{k}|^{2}dr,
\end{align*}
as desired.
\end{proof}

\begin{proof}
[Proof of Theorem \ref{T2}]If $u$ is radial, then
\begin{align*}
&  \int_{B_{R}}V_{2}(|x|)|\Delta u|^{2}dx-\int_{B_{R}}W_{2}(|x|)|\nabla
u|^{2}dx\\
&  =\int_{0}^{R}V_{2}(r)r^{N-1}\left\vert u^{\prime\prime}\right\vert
^{2}dr+(N-1)\int_{0}^{R}V_{2}(r)r^{N-3}\left\vert u^{\prime}\right\vert
^{2}dr\\
&  \hspace{0.2in}-(N-1)\int_{0}^{R}V_{2}^{\prime}\left(  r\right)
r^{N-2}\left\vert u^{\prime}\right\vert ^{2}dr-\int_{0}^{R}W_{2}%
(r)r^{N-1}\left\vert u^{\prime}\right\vert ^{2}dr.
\end{align*}
Since $\left(  V_{2},\left(  W_{2}+N\dfrac{V_{2}^{\prime}}{r}\right)  \right)
$ is a $\left(  N+2\right)  $-dimensional Bessel pair on $\left(  0,R\right)
$, we have
\[
\int_{0}^{R}V_{2}(r)r^{N+1}\left\vert \left(  \dfrac{u^{\prime}}{r}\right)
^{\prime}\right\vert ^{2}dr\geq\int_{0}^{R}\left(  W_{2}(r)+N\dfrac
{V_{2}^{\prime}(r)}{r}\right)  r^{N+1}\left\vert \dfrac{u^{\prime}}%
{r}\right\vert ^{2}dr,
\]
which is equivalent to
\begin{align*}
\int_{0}^{R}V_{2}(r)r^{N-1}\left\vert u^{\prime\prime}\right\vert ^{2}dr  &
-2\int_{0}^{R}V_{2}(r)r^{N-2}u^{\prime\prime}u^{\prime}dr+\int_{0}^{R}%
V_{2}(r)r^{N-3}\left\vert u^{\prime}\right\vert ^{2}dr\\
&  \geq\int_{0}^{R}W_{2}(r)r^{N-1}\left\vert u^{\prime}\right\vert
^{2}dr+N\int_{0}^{R}V_{2}^{\prime}\left(  r\right)  r^{N-2}\left\vert
u^{\prime}\right\vert ^{2}dr.
\end{align*}
Moreover,
\[
-2\int_{0}^{R}V_{2}(r)r^{N-2}u^{\prime\prime}u^{\prime}dr=\int_{0}^{R}%
V_{2}^{\prime}\left(  r\right)  r^{N-2}\left\vert u^{\prime}\right\vert
^{2}dr+(N-2)\int_{0}^{R}V_{2}(r)r^{N-3}\left\vert u^{\prime}\right\vert
^{2}dr.
\]
Therefore,
\begin{align*}
&  \int_{0}^{R}V_{2}(r)r^{N-1}\left\vert u^{\prime\prime}\right\vert ^{2}dr
+(N-1)\int_{0}^{R}V_{2}(r)r^{N-3}\left\vert u^{\prime}\right\vert ^{2}dr\\
&  \hspace{0.2in} -(N-1)\int_{0}^{R}V_{2}^{\prime}\left(  r\right)
r^{N-2}\left\vert u^{\prime}\right\vert ^{2}dr-\int_{0}^{R}W_{2}%
(r)r^{N-1}\left\vert u^{\prime}\right\vert ^{2}dr\geq0.
\end{align*}

Now, using the computations from the last section, we get with $u(x)=\sum
_{k=0}^{\infty}u_{k}(r)\phi_{k}(\sigma)$ that
\begin{align*}
&  \int_{B_{R}}V_{2}(|x|)|\Delta u|^{2}dx-\int_{B_{R}}W_{2}(|x|)|\nabla
u|^{2}dx\\
&  =\sum_{k=0}^{\infty}\left[  \int_{0}^{R}V_{2}(r)r^{N-1}\left\vert
u_{k}^{\prime\prime}\right\vert ^{2}dr+((N-1)+2c_{k})\int_{0}^{R}%
V_{2}(r)r^{N-3}\left\vert u_{k}^{\prime}\right\vert ^{2}dr\right. \\
&  \hspace{0.2in}+(c_{k}^{2}+2(N-4)c_{k})\int_{0}^{R}V_{2}(r)r^{N-5}\left\vert
u_{k}\right\vert ^{2}dr-(N-1)\int_{0}^{R}V_{2}^{\prime}\left(  r\right)
r^{N-2}\left\vert u_{k}^{\prime}\right\vert ^{2}dr\\
&  \hspace{0.2in}-(N-5)c_{k}\int_{0}^{R}V_{2}^{\prime}\left(  r\right)
r^{N-4}\left\vert u_{k}\right\vert ^{2}dr-c_{k}\int_{0}^{R}V_{2}^{\prime
\prime}\left(  r\right)  r^{N-3}\left\vert u_{k}\right\vert ^{2}dr\\
&  \hspace{0.2in}\left.  -\int_{0}^{R}W_{2}(r)r^{N-1}\left\vert u_{k}^{\prime
}\right\vert ^{2}dr-c_{k}\int_{0}^{R}W_{2}(r)r^{N-3}\left\vert u_{k}%
\right\vert ^{2}dr\right] \\
&  =\sum_{k=0}^{\infty}\left[  \left(  \int_{0}^{R}V_{2}(r)r^{N-1}\left\vert
u_{k}^{\prime\prime}\right\vert ^{2}dr+(N-1)\int_{0}^{R}V_{2}(r)r^{N-3}%
\left\vert u_{k}^{\prime}\right\vert ^{2}dr\right.  \right. \\
&  \hspace{0.2in}-\left.  (N-1)\int_{0}^{R}V_{2}^{\prime}\left(  r\right)
r^{N-2}\left\vert u_{k}^{\prime}\right\vert ^{2}dr-\int_{0}^{R}W_{2}%
(r)r^{N-1}\left\vert u_{k}^{\prime}\right\vert ^{2}dr\right) \\
&  \hspace{0.2in}+c_{k}\left(  2\int_{0}^{R}V_{2}(r)r^{N-3}\left\vert
u_{k}^{\prime}\right\vert ^{2}dr+(3N-9)\int_{0}^{R}V_{2}(r)r^{N-5}\left\vert
u_{k}\right\vert ^{2}dr\right. \\
&  \hspace{0.2in}-(N-5)\int_{0}^{R}V_{2}^{\prime}\left(  r\right)
r^{N-4}\left\vert u_{k}\right\vert ^{2}dr-\int_{0}^{R}W_{2}(r)r^{N-3}%
\left\vert u_{k}\right\vert ^{2}dr\\
&  \hspace{0.2in}-\left.  \left.  \int_{0}^{R}V_{2}^{\prime\prime}\left(
r\right)  r^{N-3}\left\vert u_{k}\right\vert ^{2}dr+(c_{k}-(N-1))\int_{0}%
^{R}V_{2}(r)r^{N-5}\left\vert u_{k}\right\vert ^{2}dr\right)  \right]  .
\end{align*}
Since each $u_{k}$ is radial, we obtain from above that
\begin{align}
&  \int_{0}^{R}V_{2}(r)r^{N-1}\left\vert u_{k}^{\prime\prime}\right\vert
^{2}dr +(N-1)\int_{0}^{R}V_{2}(r)r^{N-3}\left\vert u_{k}^{\prime}\right\vert
^{2}dr\nonumber\\
&  \hspace{0.2in} -(N-1)\int_{0}^{R}V_{2}^{\prime}\left(  r\right)
r^{N-2}\left\vert u_{k}^{\prime}\right\vert ^{2}dr-\int_{0}^{R}W_{2}%
(r)r^{N-1}\left\vert u_{k}^{\prime}\right\vert ^{2}dr\geq0. \label{eq 2.5}%
\end{align}
From \eqref{eq 2.5}, the condition \eqref{ConM} and the fact that $c_{k}%
\geq(N-1)$ $\forall k\geq1$, we achieve the desired result.
\end{proof}

\begin{proof}
[Proof of Theorem \ref{T2a}]To show Theorem \ref{T2a}, it suffices to show
that condition \eqref{Con2} implies condition \eqref{ConM}.

Indeed, by setting $u=r^{2}v$, we get
\begin{align*}
&  \int_{0}^{R}V_{2}(r)r^{N-3}\left\vert u^{\prime}\right\vert ^{2}dr\\
&  =4\int_{0}^{R}V_{2}(r)r^{N-1}\left\vert v\right\vert ^{2}dr+\int_{0}%
^{R}V_{2}(r)r^{N+1}\left\vert v^{\prime}\right\vert ^{2}dr+4\int_{0}^{R}%
V_{2}(r)r^{N}vv^{\prime}dr.\\
&  =(4-2N)\int_{0}^{R}V_{2}(r)r^{N-1}\left\vert v\right\vert ^{2}dr-2\int
_{0}^{R}V_{2}^{\prime}(r)r^{N}\left\vert v\right\vert ^{2}dr+\int_{0}^{R}%
V_{2}(r)r^{N+1}\left\vert v^{\prime}\right\vert ^{2}dr\\
&  =(4-2N)\int_{0}^{R}V_{2}(r)r^{N-5}\left\vert u\right\vert ^{2}dr-2\int
_{0}^{R}V_{2}^{\prime}(r)r^{N-4}\left\vert u\right\vert ^{2}dr+\int_{0}%
^{R}V_{2}(r)r^{N+1}\left\vert v^{\prime}\right\vert ^{2}dr.
\end{align*}
Thus, since $\left(  V_{2},\left(  W_{2}+N\dfrac{V_{2}^{\prime}}{r}\right)
\right)  $ is a $\left(  N+2\right)  $-dimensional Bessel pair on $\left(
0,R\right)  $, we have
\begin{align*}
\int_{0}^{R}V_{2}(r)r^{N-3}\left\vert u^{\prime}\right\vert ^{2}dr  &
+(2N-4)\int_{0}^{R}V_{2}(r)r^{N-5}\left\vert u\right\vert ^{2}dr+2\int_{0}%
^{R}V_{2}^{\prime}(r)r^{N-4}\left\vert u\right\vert ^{2}dr\\
&  =\int_{0}^{R}V_{2}(r)r^{N+1}\left\vert v^{\prime}\right\vert ^{2}dr\\
&  \geq\int_{0}^{R}\left(  W_{2}(r)+N\dfrac{V_{2}^{\prime}(r)}{r}\right)
r^{N+1}\left\vert v\right\vert ^{2}dr\\
&  \geq\int_{0}^{R}W_{2}(r)r^{N-3}\left\vert u\right\vert ^{2}dr+N\int_{0}%
^{R}V_{2}^{\prime}(r)r^{N-4}\left\vert u\right\vert ^{2}dr,
\end{align*}
and
\begin{align*}
&  \int_{0}^{R}V_{2}(r)r^{N-3}\left\vert u^{\prime}\right\vert ^{2}%
dr+(2N-4)\int_{0}^{R}V_{2}(r)r^{N-5}\left\vert u\right\vert ^{2}dr\\
&  \hspace{0.2in}+(2-N)\int_{0}^{R}V_{2}^{\prime}(r)r^{N-4}\left\vert
u\right\vert ^{2}dr-\int_{0}^{R}W_{2}(r)r^{N-3}\left\vert u\right\vert
^{2}dr\geq0.
\end{align*}
Therefore, if (\ref{Con2}) holds, then%
\begin{align*}
&  2\int_{0}^{R}V_{2}(r)r^{N-3}\left\vert u^{\prime}\right\vert ^{2}%
dr-\int_{0}^{R}V_{2}^{\prime\prime}\left(  r\right)  r^{N-3}\left\vert
u\right\vert ^{2}-\left(  N-5\right)  \int_{0}^{R}V_{2}^{\prime}\left(
r\right)  r^{N-4}\left\vert u\right\vert ^{2}dr\\
&  \hspace{0.2in} +(3N-9)\int_{0}^{R}V_{2}(r)r^{N-5}\left\vert u\right\vert
^{2}dr-\int_{0}^{R}W_{2}(r)r^{N-3}\left\vert u\right\vert ^{2}dr\\
&  =\left[  \int_{0}^{R}V_{2}(r)r^{N-3}\left\vert u^{\prime}\right\vert
^{2}dr+(2N-4)\int_{0}^{R}V_{2}(r)r^{N-5}\left\vert u\right\vert ^{2}dr\right.
\\
&  \hspace{0.2in} \left.  +(2-N)\int_{0}^{R}V_{2}^{\prime}(r)r^{N-4}\left\vert
u\right\vert ^{2}dr-\int_{0}^{R}W_{2}(r)r^{N-3}\left\vert u\right\vert
^{2}dr\right] \\
&  \hspace{0.2in} +\left[  \int_{0}^{R}V_{2}(r)r^{N-3}\left\vert u^{\prime
}\right\vert ^{2}dr+(N-5)\int_{0}^{R}V_{2}(r)r^{N-5}\left\vert u\right\vert
^{2}dr\right. \\
&  \hspace{0.2in} \left.  -\int_{0}^{R}V_{2}^{\prime\prime}\left(  r\right)
r^{N-3}\left\vert u\right\vert ^{2}dr+3\int_{0}^{R}V_{2}^{\prime}\left(
r\right)  r^{N-4}\left\vert u\right\vert ^{2}dr\right]  \geq0.
\end{align*}

This shows that condition \eqref{Con2} implies condition \eqref{ConM}.
\end{proof}

\begin{proof}
[Proof of Corollary \ref{C1}]Since $\left(  V_{1},W_{1}\right)  $ is a
$N$-dimensional Bessel pair on $\left(  0,R\right)  $, there exists a positive
function $\varphi$ such that
\[
\left(  r^{N-1}V_{1}\varphi^{\prime}\right)  ^{\prime}+r^{N-1}W_{1}%
\varphi=0\text{ on }\left(  0,R\right)  \text{.}%
\]
Let $\varphi=r\psi$. Then $\varphi^{\prime}=\psi+r\psi^{\prime}$ and
$\varphi^{\prime\prime}=2\psi^{\prime}+r\psi^{\prime\prime}$. Therefore, from%
\[
\varphi^{\prime\prime}+\left(  \frac{N-1}{r}+\frac{V_{1}^{\prime}}{V_{1}%
}\right)  \varphi^{\prime}+\frac{W_{1}}{V_{1}}\varphi=0\text{ on }\left(
0,R\right)  ,
\]
we obtain%
\[
2\psi^{\prime}+r\psi^{\prime\prime}+\left(  \frac{N-1}{r}+\frac{V_{1}^{\prime
}}{V_{1}}\right)  \left(  \psi+r\psi^{\prime}\right)  +\frac{W_{1}}{V_{1}%
}r\psi=0\text{ on }\left(  0,R\right)  .
\]
That is%
\[
\psi^{\prime\prime}+\left(  \frac{N+1}{r}+\frac{V_{1}^{\prime}}{V_{1}}\right)
\psi^{\prime}+\left(  \frac{N-1}{r^{2}}+\frac{V_{1}^{\prime}}{rV_{1}}%
+\frac{W_{1}}{V_{1}}\right)  \psi=0\text{ on }\left(  0,R\right)  .
\]
Hence $\left(  V_{1},\left(  W_{2}+N\dfrac{V_{1}^{\prime}}{r}\right)  \right)
$ is a $\left(  N+2\right)  $-dimensional Bessel pair on $\left(  0,R\right)
$ with $W_{2}=\frac{N-1}{r^{2}}V_{1}-\left(  N-1\right)  \frac{V_{1}^{\prime}%
}{r}+W_{1}$. Therefore, Theorem \ref{T2} yields%
\begin{align*}
\int_{B_{R}}V_{1}(|x|)|\Delta u|^{2}dx  &  \geq\int_{B_{R}}W_{1}(|x|)|\nabla
u|^{2}dx\\
&  +\left(  N-1\right)  \int_{B_{R}}\left(  \frac{V_{1}\left(  \left\vert
x\right\vert \right)  }{\left\vert x\right\vert ^{2}}-\frac{V_{1}^{\prime
}\left(  \left\vert x\right\vert \right)  }{\left\vert x\right\vert }\right)
\left\vert \nabla u\right\vert ^{2}dx.
\end{align*}

\end{proof}

\begin{proof}
[Proof of Theorem \ref{T3}]The case $N=1$ is obvious, so now we assume that
$N\geq2$. Let $u\in C_{0}^{\infty}\left(  B_{R}\setminus\left\{  0\right\}
\right)  $. Then, using the spherical harmonics decomposition $u(x)=\sum
_{k=0}^{\infty}u_{k}(r)\phi_{k}(\sigma)$, we get
\begin{align*}
&  \int_{B_{R}}V_{2}(|x|)|\Delta u|^{2}dx-\int_{B_{R}}W_{2}(|x|)|\nabla
u|^{2}dx\\
&  =\sum_{k=0}^{\infty}\left[  \int_{0}^{R}V_{2}(r)r^{N-1}\left\vert
u_{k}^{\prime\prime}\right\vert ^{2}dr+((N-1)+2c_{k})\int_{0}^{R}%
V_{2}(r)r^{N-3}\left\vert u_{k}^{\prime}\right\vert ^{2}dr\right. \\
&  \hspace{0.2in}+(c_{k}^{2}+2(N-4)c_{k})\int_{0}^{R}V_{2}(r)r^{N-5}\left\vert
u_{k}\right\vert ^{2}dr\\
&  \hspace{0.2in}-(N-1)\int_{0}^{R}V_{2}^{\prime}\left(  r\right)
r^{N-2}\left\vert u_{k}^{\prime}\right\vert ^{2}dr-(N-5)c_{k}\int_{0}^{R}%
V_{2}^{\prime}\left(  r\right)  r^{N-4}\left\vert u_{k}\right\vert ^{2}dr\\
&  \hspace{0.2in}\left.  -~c_{k}\int_{0}^{R}V_{2}^{\prime\prime}\left(
r\right)  r^{N-3}\left\vert u_{k}\right\vert ^{2}dr-\int_{0}^{R}%
W_{2}(r)r^{N-1}\left\vert u_{k}^{\prime}\right\vert ^{2}dr-c_{k}\int_{0}%
^{R}W_{2}(r)r^{N-3}\left\vert u_{k}\right\vert ^{2}dr\right]  .
\end{align*}

Therefore
\[
\int_{\mathbb{R}^{N}}V_{2}\left(  \left\vert x\right\vert \right)  |\Delta
u|^{2}dx\geq\int_{\mathbb{R}^{N}}W_{2}\left(  \left\vert x\right\vert \right)
|\nabla u|^{2}dx
\]
is equivalent to%
\begin{align*}
&  \int_{0}^{R}V_{2}\left(  r\right)  r^{N-1}\left\vert u_{k}^{\prime\prime
}\right\vert ^{2}dr+(N-1+2c_{k})\int_{0}^{R}V_{2}\left(  r\right)
r^{N-3}\left\vert u_{k}^{\prime}\right\vert ^{2}dr\\
&  \hspace{0.2in}+(c_{k}^{2}+2(N-4)c_{k})\int_{0}^{R}V_{2}\left(  r\right)
r^{N-5}|u_{k}|^{2}dr-\left(  N-1\right)  \int_{0}^{R}V_{2}^{\prime}\left(
r\right)  r^{N-2}\left\vert u_{k}^{\prime}\right\vert ^{2}dr\\
&  \hspace{0.2in}-\left(  N-5\right)  c_{k}\int_{0}^{R}V_{2}^{\prime}\left(
r\right)  r^{N-4}\left\vert u_{k}\right\vert ^{2}dr-c_{k}\int_{0}^{R}%
V_{2}^{\prime\prime}\left(  r\right)  r^{N-3}\left\vert u_{k}\right\vert
^{2}dr\\
&  \geq\left(  \int_{0}^{R}W_{2}\left(  r\right)  r^{N-1}\left\vert u^{\prime
}\right\vert ^{2}dr+c_{k}\int_{0}^{R}W_{2}\left(  r\right)  r^{N-3}\left\vert
u\right\vert ^{2}dr\right)  \text{ for all }k\geq0\text{.}%
\end{align*}
So now, we will need to prove that
\begin{align}
\int_{0}^{R}V_{2}(r)r^{N-1}\left\vert u^{\prime\prime}\right\vert
^{2}dr+(N-1)\int_{0}^{R}V_{2}(r)r^{N-3}\left\vert u^{\prime}\right\vert
^{2}dr  &  -(N-1)\int_{0}^{R}V_{2}^{\prime}\left(  r\right)  r^{N-2}\left\vert
u^{\prime}\right\vert ^{2}dr\nonumber\\
&  \geq\int_{0}^{R}W_{2}(r)r^{N-1}\left\vert u^{\prime}\right\vert ^{2}dr,
\label{n0}%
\end{align}
and%
\begin{align}
&  \int_{0}^{R}V_{2}(r)r^{N-1}\left\vert u^{\prime\prime}\right\vert
^{2}dr+((N-1)+2c_{1})\int_{0}^{R}V_{2}(r)r^{N-3}\left\vert u^{\prime
}\right\vert ^{2}dr\nonumber\\
&  \hspace{0.2in}+(c_{1}^{2}+2(N-4)c_{1})\int_{0}^{R}V_{2}(r)r^{N-5}\left\vert
u\right\vert ^{2}dr-(N-1)\int_{0}^{R}V_{2}^{\prime}\left(  r\right)
r^{N-2}\left\vert u^{\prime}\right\vert ^{2}dr\nonumber\\
&  \hspace{0.2in}-(N-5)c_{1}\int_{0}^{R}V_{2}^{\prime}\left(  r\right)
r^{N-4}\left\vert u\right\vert ^{2}dr-c_{1}\int_{0}^{R}V_{2}^{\prime\prime
}\left(  r\right)  r^{N-3}\left\vert u\right\vert ^{2}dr\nonumber\\
&  \geq\int_{0}^{R}W_{2}(r)r^{N-1}\left\vert u^{\prime}\right\vert
^{2}dr+c_{1}\int_{0}^{R}W_{2}(r)r^{N-3}\left\vert u\right\vert ^{2}dr,
\label{n1}%
\end{align}
for all $u\in C_{0}^{\infty}\left(  0,R\right)  $ imply%
\begin{align*}
&  \int_{0}^{R}V_{2}\left(  r\right)  r^{N-1}\left\vert u^{\prime\prime
}\right\vert ^{2}dr+(N-1+2c_{k})\int_{0}^{R}V_{2}\left(  r\right)
r^{N-3}\left\vert u^{\prime}\right\vert ^{2}dr\\
&  \hspace{0.2in}+(c_{k}^{2}+2(N-4)c_{k})\int_{0}^{R}V_{2}\left(  r\right)
r^{N-5}|u|^{2}dr-\left(  N-1\right)  \int_{0}^{R}V_{2}^{\prime}\left(
r\right)  r^{N-2}\left\vert u^{\prime}\right\vert ^{2}dr\\
&  \hspace{0.2in}-\left(  N-5\right)  c_{k}\int_{0}^{R}V_{2}^{\prime}\left(
r\right)  r^{N-4}\left\vert u\right\vert ^{2}dr-c_{k}\int_{0}^{R}V_{2}%
^{\prime\prime}\left(  r\right)  r^{N-3}\left\vert u\right\vert ^{2}dr\\
&  \geq\left(  \int_{0}^{R}W_{2}\left(  r\right)  r^{N-1}\left\vert u^{\prime
}\right\vert ^{2}dr+c_{k}\int_{0}^{R}W_{2}\left(  r\right)  r^{N-3}\left\vert
u\right\vert ^{2}dr\right)
\end{align*}
for all $k\geq2$ and $u\in C_{0}^{\infty}\left(  0,R\right)  .$

Because of (\ref{n1}), we just need to prove that for all $k\geq2$ and $u\in
C_{0}^{\infty}\left(  0,R\right)  :$
\begin{align*}
&  2\left(  c_{k}-c_{1}\right)  \int_{0}^{R}V_{2}\left(  r\right)
r^{N-3}\left\vert u^{\prime}\right\vert ^{2}dr+(c_{k}^{2}+2(N-4)c_{k}%
-c_{1}^{2}-2(N-4)c_{1})\int_{0}^{R}V_{2}\left(  r\right)  r^{N-5}|u|^{2}dr\\
&  \hspace{0.2in}-\left(  N-5\right)  \left(  c_{k}-c_{1}\right)  \int_{0}%
^{R}V_{2}^{\prime}\left(  r\right)  r^{N-4}\left\vert u\right\vert
^{2}dr-\left(  c_{k}-c_{1}\right)  \int_{0}^{R}V_{2}^{\prime\prime}\left(
r\right)  r^{N-3}\left\vert u\right\vert ^{2}dr\\
&  \geq\left(  c_{k}-c_{1}\right)  \int_{0}^{R}W_{2}\left(  r\right)
r^{N-3}\left\vert u\right\vert ^{2}dr.
\end{align*}
Equivalently%
\begin{align*}
&  2\int_{0}^{R}V_{2}\left(  r\right)  r^{N-3}\left\vert u^{\prime}\right\vert
^{2}dr+\left(  c_{k}+3N-9\right)  \int_{0}^{R}V_{2}\left(  r\right)
r^{N-5}|u|^{2}dr\\
&  \hspace{0.2in}-\left(  N-5\right)  \int_{0}^{R}V_{2}^{\prime}\left(
r\right)  r^{N-4}\left\vert u\right\vert ^{2}dr-\int_{0}^{R}V_{2}%
^{\prime\prime}\left(  r\right)  r^{N-3}\left\vert u\right\vert ^{2}dr\\
&  \geq\int_{0}^{R}W_{2}\left(  r\right)  r^{N-3}\left\vert u\right\vert
^{2}dr.
\end{align*}
Now, by noting that when $k\geq2$, $c_{k}\geq2N$, we obtain our desired result.
\end{proof}

\begin{proof}
[Proof of Theorem \ref{T3a}]

Indeed, thanks to (node $k=0$), we have
\begin{align*}
&  \int_{0}^{R}V_{2}(r)r^{N-1}\left\vert \left(  \frac{u}{r}\right)  ^{\prime
}\right\vert ^{2}dr+(N-1)\int_{0}^{R}V_{2}(r)r^{N-3}\left\vert \frac{u}%
{r}\right\vert ^{2}dr-(N-1)\int_{0}^{R}V_{2}^{\prime}\left(  r\right)
r^{N-2}\left\vert \frac{u}{r}\right\vert ^{2}dr\\
&  =\int_{0}^{R}V_{2}(r)r^{N-3}\left\vert u^{\prime}\right\vert ^{2}%
dr+(2N-4)\int_{0}^{R}V_{2}(r)r^{N-5}\left\vert u\right\vert ^{2}%
dr-(N-2)\int_{0}^{R}V_{2}^{\prime}\left(  r\right)  r^{N-4}\left\vert
u\right\vert ^{2}dr\\
&  \geq\int_{0}^{R}W_{2}(r)r^{N-3}\left\vert u\right\vert ^{2}dr.
\end{align*}
Therefore,
\begin{align*}
&  \int_{0}^{R}V_{2}\left(  r\right)  r^{N-3}\left\vert u^{\prime}\right\vert
^{2}dr+\left(  3N-5\right)  \int_{0}^{R}V_{2}\left(  r\right)  r^{N-5}%
|u|^{2}dr\\
&  \hspace{0.2in}+\int_{0}^{R}\left(  3\frac{V_{2}^{\prime}(r)}{r}%
-V_{2}^{\prime\prime}(r)\right)  r^{N-3}\left\vert u\right\vert ^{2}dr\\
&  \hspace{0.2in}+\int_{0}^{R}V_{2}\left(  r\right)  r^{N-3}\left\vert
u^{\prime}\right\vert ^{2}dr+(2N-4)\int_{0}^{R}V_{2}\left(  r\right)
r^{N-5}|u|^{2}dr\\
&  \hspace{0.2in}-(N-2)\int_{0}^{R}V_{2}^{\prime}\left(  r\right)
r^{N-4}\left\vert u\right\vert ^{2}dr\\
&  \geq\int_{0}^{R}W_{2}(r)r^{N-3}\left\vert u\right\vert ^{2}dr.
\end{align*}

\end{proof}

\begin{proof}
[Proof of Theorem \ref{T3.1}]From Theorem \ref{T3}, it is now enough to show
that
\begin{align}
\int_{0}^{R}V_{2}(r)r^{N-1}\left\vert u^{\prime\prime}\right\vert
^{2}dr+(N-1)\int_{0}^{R}V_{2}(r)r^{N-3}\left\vert u^{\prime}\right\vert
^{2}dr  &  -(N-1)\int_{0}^{R}V_{2}^{\prime}\left(  r\right)  r^{N-2}\left\vert
u^{\prime}\right\vert ^{2}dr\nonumber\\
&  \geq\int_{0}^{R}W_{2}(r)r^{N-1}\left\vert u^{\prime}\right\vert ^{2}dr,
\label{cond2.4}%
\end{align}
for all $u\in C_{0}^{\infty}\left(  0,R\right)  $ implies%
\begin{align*}
&  \int_{0}^{R}V_{2}(r)r^{N-1}\left\vert u^{\prime\prime}\right\vert
^{2}dr+((N-1)+2c_{1})\int_{0}^{R}V_{2}(r)r^{N-3}\left\vert u^{\prime
}\right\vert ^{2}dr\\
&  \hspace{0.2in}+(c_{1}^{2}+2(N-4)c_{1})\int_{0}^{R}V_{2}(r)r^{N-5}\left\vert
u\right\vert ^{2}dr-(N-1)\int_{0}^{R}V_{2}^{\prime}\left(  r\right)
r^{N-2}\left\vert u^{\prime}\right\vert ^{2}dr\\
&  \hspace{0.2in}-(N-5)c_{1}\int_{0}^{R}V_{2}^{\prime}\left(  r\right)
r^{N-4}\left\vert u\right\vert ^{2}dr-c_{1}\int_{0}^{R}V_{2}^{\prime\prime
}\left(  r\right)  r^{N-3}\left\vert u\right\vert ^{2}dr\\
&  \geq\int_{0}^{R}W_{2}(r)r^{N-1}\left\vert u^{\prime}\right\vert
^{2}dr+c_{1}\int_{0}^{R}W_{2}(r)r^{N-3}\left\vert u\right\vert ^{2}dr
\end{align*}
for all $u\in C_{0}^{\infty}\left(  0,R\right)  $.

Indeed, assume \eqref{cond2.4}, we will now prove that%
\begin{align*}
&  2c_{1}\int_{0}^{R}V_{2}(r)r^{N-3}\left\vert u^{\prime}\right\vert
^{2}dr+(c_{1}^{2}+2(N-4)c_{1})\int_{0}^{R}V_{2}(r)r^{N-5}\left\vert
u\right\vert ^{2}dr\\
&  \hspace{0.2in}-(N-5)c_{1}\int_{0}^{R}V_{2}^{\prime}\left(  r\right)
r^{N-4}\left\vert u\right\vert ^{2}dr-c_{1}\int_{0}^{R}V_{2}^{\prime\prime
}\left(  r\right)  r^{N-3}\left\vert u\right\vert ^{2}dr\\
&  \geq c_{1}\int_{0}^{R}W_{2}(r)r^{N-3}\left\vert u\right\vert ^{2}dr.
\end{align*}
Equivalently,%
\begin{align*}
&  2\int_{0}^{R}V_{2}(r)r^{N-3}\left\vert u^{\prime}\right\vert ^{2}%
dr+(3N-9)\int_{0}^{R}V_{2}(r)r^{N-5}\left\vert u\right\vert ^{2}dr\\
&  \hspace{0.2in}-(N-5)\int_{0}^{R}V_{2}^{\prime}\left(  r\right)
r^{N-4}\left\vert u\right\vert ^{2}dr-\int_{0}^{R}V_{2}^{\prime\prime}\left(
r\right)  r^{N-3}\left\vert u\right\vert ^{2}dr\\
&  \geq\int_{0}^{R}W_{2}(r)r^{N-3}\left\vert u\right\vert ^{2}dr.
\end{align*}
This is exactly the constraint (\ref{ConM}).

If $V_{2}$ sastisfies (\ref{Con2}), that is
\[
V_{2}^{\prime\prime}(r)-3\dfrac{V_{2}^{\prime}(r)}{r}-\left(  N-5\right)
\dfrac{V_{2}(r)}{r^{2}}\leq0\text{ on }\left(  0,R\right)  \text{,}%
\]
then from \eqref{cond2.4}, we deduce that
\begin{align*}
\int_{0}^{R}V_{2}(r)r^{N-1}\left\vert \left(  \frac{u}{r}\right)  ^{\prime
}\right\vert ^{2}dr+(N-1)\int_{0}^{R}V_{2}(r)r^{N-3}\left\vert \frac{u}%
{r}\right\vert ^{2}dr  &  -(N-1)\int_{0}^{R}V_{2}^{\prime}\left(  r\right)
r^{N-2}\left\vert \frac{u}{r}\right\vert ^{2}dr\\
&  \geq\int_{0}^{R}W_{2}(r)r^{N-1}\left\vert \frac{u}{r}\right\vert ^{2}dr.
\end{align*}
Equivalently%
\begin{align*}
&  \int_{0}^{R}V_{2}(r)r^{N-3}\left\vert u^{\prime}\right\vert ^{2}%
dr+(2N-4)\int_{0}^{R}V_{2}(r)r^{N-5}\left\vert u\right\vert ^{2}dr\\
&  \hspace{0.2in}+(2-N)\int_{0}^{R}V_{2}^{\prime}(r)r^{N-4}\left\vert
u\right\vert ^{2}dr-\int_{0}^{R}W_{2}(r)r^{N-3}\left\vert u\right\vert
^{2}dr\geq0.
\end{align*}
Therefore,
\begin{align*}
&  2\int_{0}^{R}V_{2}(r)r^{N-3}\left\vert u^{\prime}\right\vert ^{2}%
dr-\int_{0}^{R}V_{2}^{\prime\prime}\left(  r\right)  r^{N-3}\left\vert
u\right\vert ^{2}-\left(  N-5\right)  \int_{0}^{R}V_{2}^{\prime}\left(
r\right)  r^{N-4}\left\vert u\right\vert ^{2}dr\\
&  \hspace{0.2in}+(3N-9)\int_{0}^{R}V_{2}(r)r^{N-5}\left\vert u\right\vert
^{2}dr-\int_{0}^{R}W_{2}(r)r^{N-3}\left\vert u\right\vert ^{2}dr\\
&  =\left[  \int_{0}^{R}V_{2}(r)r^{N-3}\left\vert u^{\prime}\right\vert
^{2}dr+(2N-4)\int_{0}^{R}V_{2}(r)r^{N-5}\left\vert u\right\vert ^{2}dr\right.
\\
&  \hspace{0.2in}\left.  +(2-N)\int_{0}^{R}V_{2}^{\prime}(r)r^{N-4}\left\vert
u\right\vert ^{2}dr-\int_{0}^{R}W_{2}(r)r^{N-3}\left\vert u\right\vert
^{2}dr\right] \\
&  \hspace{0.2in}+\left[  \int_{0}^{R}V_{2}(r)r^{N-3}\left\vert u^{\prime
}\right\vert ^{2}dr+(N-5)\int_{0}^{R}V_{2}(r)r^{N-5}\left\vert u\right\vert
^{2}dr\right. \\
&  \hspace{0.2in}\left.  -\int_{0}^{R}V_{2}^{\prime\prime}\left(  r\right)
r^{N-3}\left\vert u\right\vert ^{2}dr+3\int_{0}^{R}V_{2}^{\prime}\left(
r\right)  r^{N-4}\left\vert u\right\vert ^{2}dr\right]  \geq0.
\end{align*}

\end{proof}

\begin{proof}
[Proof of Theorem \ref{T4}]As in the proof of Theorem \ref{T3}, we get%
\begin{align*}
&  2\int_{0}^{R}r^{N-3}\left\vert u^{\prime}\right\vert ^{2}dr+\left(
c_{k}+3N-9\right)  \int_{0}^{R}r^{N-5}|u|^{2}dr\\
&  \geq\int_{0}^{R}r^{N-3}\left\vert u^{\prime}\right\vert ^{2}dr+\left(
3N-5\right)  \int_{0}^{R}r^{N-5}|u|^{2}dr\\
&  \hspace{0.2in}+\int_{0}^{R}r^{N-3}\left\vert u^{\prime}\right\vert
^{2}dr+(2N-4)\int_{0}^{R}r^{N-5}|u|^{2}dr\\
&  \geq\left(  \frac{\left(  N-4\right)  ^{2}}{4}+3N-5\right)  \int_{0}%
^{R}r^{N-5}|u|^{2}dr\\
&  \hspace{0.2in}+\int_{0}^{R}r^{N-3}\left\vert u^{\prime}\right\vert
^{2}dr+(2N-4)\int_{0}^{R}r^{N-5}|u|^{2}dr\\
&  \geq\int_{0}^{R}W_{2}(r)r^{N-3}\left\vert u\right\vert ^{2}dr.
\end{align*}

\end{proof}

\begin{proof}
[Proof of Theorem \ref{T5.1}]We note that
\begin{equation}
\int_{B_{R}}|\Delta u|^{2}dx\geq\int_{B_{R}}W_{2}(|x|)|\nabla u|^{2}dx.
\label{HR3}%
\end{equation}
for a radial function$\;u\in C_{0}^{\infty}(B_{R}\cap\mathbb{R}^{N}%
\setminus\left\{  0\right\}  )$ is equivalent to%
\[
\int_{0}^{R}r^{N-1}\left\vert u^{\prime\prime}\right\vert ^{2}dr+(N-1)\int
_{0}^{R}r^{N-3}\left\vert u^{\prime}\right\vert ^{2}dr\geq\int_{0}%
^{R}W(r)r^{N-1}|u^{\prime}|^{2}dr,
\]
which is now equivalent to
\begin{equation}
\int_{0}^{R}\left\vert v^{\prime}\right\vert ^{2}r^{N+1}dr\geq\int_{0}%
^{R}W_{2}\left(  r\right)  \left\vert v\right\vert ^{2}r^{N+1}dr \label{H1}%
\end{equation}
with $u^{\prime}\left(  r\right)  =rv\left(  r\right)  $. Therefore, by
Theorem 4.1.1 in \cite{GM1}, we deduce that there exists $c>0$ such that
$\left(  1,cW_{2}\right)  $ is a $\left(  N+2\right)  $-dimensional Bessel
pair on $\left(  0,R\right)  $.
\end{proof}

\begin{proof}
[Proof of Theorem \ref{T6}]Let $u\in C_{0}^{\infty}(B_{R}\setminus\left\{
0\right\}  )$. Then, using the spherical harmonics decomposition
$u(x)=\sum_{k=0}^{\infty}u_{k}(r)\phi_{k}(\sigma)$, we get
\begin{align*}
&  \int_{B_{R}}|\Delta u|^{2}dx-\int_{B_{R}}W(|x|)|u|^{2}dx\\
&  =\sum_{k=0}^{\infty}\left[  \int_{0}^{R}r^{N-1}\left\vert u_{k}%
^{\prime\prime}\right\vert ^{2}dr+((N-1)+2c_{k})\int_{0}^{\infty}%
r^{N-3}\left\vert u_{k}^{\prime}\right\vert ^{2}dr\right. \\
&  \hspace{0.2in}\left.  +(c_{k}^{2}+2(N-4)c_{k})\int_{0}^{\infty}%
r^{N-5}\left\vert u_{k}\right\vert ^{2}dr-\int_{0}^{\infty}W(r)r^{N-1}%
|u_{k}|^{2}dr\right]  .
\end{align*}
Since $\{u_{k}\}_{k=0}^{\infty}$ are independent, the forward direction of
theorem is obvious. As a result, we just need to consider the backward
direction. That is, it suffices to prove that%
\begin{align*}
&  \int_{0}^{R}r^{N-1}\left\vert v^{\prime\prime}\right\vert ^{2}%
dr+(N-1)\int_{0}^{R}r^{N-3}\left\vert v^{\prime}\right\vert ^{2}dr\\
&  \hspace{0.2in}-\int_{0}^{R}W(r)r^{N-1}|v|^{2}dr\geq0\text{ }\forall v\in
C_{0}^{\infty}\left(  0,R\right)
\end{align*}
implies
\begin{align*}
\int_{0}^{R}r^{N-1}\left\vert v^{\prime\prime}\right\vert ^{2}dr  &
+((N-1)+2c_{k})\int_{0}^{R}r^{N-3}\left\vert v^{\prime}\right\vert
^{2}dr+(c_{k}^{2}+2(N-4)c_{k})\int_{0}^{R}r^{N-5}|v|^{2}dr\\
&  -\int_{0}^{R}W(r)r^{N-1}|v|^{2}dr\geq0\text{ }\forall v\in C_{0}^{\infty
}\left(  0,R\right)
\end{align*}
for all $k\geq1$.

This is obvious since for $N\geq3$ and $c_{k}\geq N-1$ $\forall k\geq1$, we
have%
\begin{align*}
&  \int_{0}^{R}r^{N-1}\left\vert v^{\prime\prime}\right\vert ^{2}%
dr+((N-1)+2c_{k})\int_{0}^{R}r^{N-3}\left\vert v^{\prime}\right\vert ^{2}dr\\
&  \hspace{0.2in}+(c_{k}^{2}+2(N-4)c_{k})\int_{0}^{R}r^{N-5}|v|^{2}dr-\int
_{0}^{R}W(r)r^{N-1}|v|^{2}dr\\
&  =\int_{0}^{R}r^{N-1}\left\vert v^{\prime\prime}\right\vert ^{2}%
dr+(N-1)\int_{0}^{R}r^{N-3}\left\vert v^{\prime}\right\vert ^{2}dr-\int
_{0}^{R}W(r)r^{N-1}|v|^{2}dr\\
&  \hspace{0.2in}+c_{k}\left[  2\int_{0}^{R}r^{N-3}\left\vert v^{\prime
}\right\vert ^{2}dr+\left(  c_{k}+2N-8\right)  \int_{0}^{R}r^{N-5}%
|v|^{2}dr\right] \\
&  \geq\int_{0}^{R}r^{N-1}\left\vert v^{\prime\prime}\right\vert
^{2}dr+(N-1)\int_{0}^{R}r^{N-3}\left\vert v^{\prime}\right\vert ^{2}%
dr-\int_{0}^{R}W(r)r^{N-1}|v|^{2}dr\\
&  \hspace{0.2in}+c_{k}\left[  2\int_{0}^{R}r^{N-3}\left\vert v^{\prime
}\right\vert ^{2}dr+3\left(  N-3\right)  \int_{0}^{R}r^{N-5}|v|^{2}dr\right]
\\
&  \geq\int_{0}^{R}r^{N-1}\left\vert v^{\prime\prime}\right\vert
^{2}dr+(N-1)\int_{0}^{R}r^{N-3}\left\vert v^{\prime}\right\vert ^{2}%
dr-\int_{0}^{R}W(r)r^{N-1}|v|^{2}dr.
\end{align*}

\end{proof}


\begin{thebibliography}{99}                                                                                               %


\bibitem {BEL}Balinsky, A. A.; Evans, W. D.; Lewis, R. T. The analysis and
geometry of Hardy's inequality. Universitext. Springer, Cham, 2015. xv+263 pp.

\bibitem {BFT04} Barbatis, G.; Filippas, S.; Tertikas, A. A unified approach
to improved Lp Hardy inequalities with best constants. Trans. Amer. Math. Soc.
356 (2004), no. 6, 2169--2196.

\bibitem {Bec08}Beckner, W. Weighted inequalities and Stein-Weiss potentials.
Forum Math. 20 (2008), no. 4, 587--606.

\bibitem {BGGP20}Berchio, E.; Ganguly, D.; Grillo, G.; Pinchover, Y. An
optimal improvement for the Hardy inequality on the hyperbolic space and
related manifolds. Proc. Roy. Soc. Edinburgh Sect. A 150 (2020), no. 4, 1699--1736.

\bibitem {BGR22}Berchio, E.; Ganguly, D.; Roychowdhury, P. Hardy-Rellich and
second order Poincar\'{e} identities on the hyperbolic space via Bessel pairs.
Calc. Var. Partial Differential Equations 61 (2022), no. 4, Paper No. 130, 24 pp.

\bibitem {BV97}Brezis, H.; V\'{a}zquez, J. L. Blow-up solutions of some
nonlinear elliptic problems. Rev. Mat. Univ. Complut. Madrid 10 (1997), no. 2, 443--469.

\bibitem {CKN}Caffarelli, L.; Kohn, R.; Nirenberg, L. First order
interpolation inequalities with weights. Compositio Math. 53 (1984), no. 3, 259--275.

\bibitem {CC09}Catrina, F.; Costa, D. G. Sharp weighted-norm inequalities for
functions with compact support in $\mathbb{R}^{N}\setminus\left\{  0\right\}
$. J. Differential Equations 246 (2009), no. 1, 164--182.

\bibitem {CFL23}Cazacu, C.; Flynn, J.; Lam, N. Caffarelli-Kohn-Nirenberg
inequalities for curl-free vector fields and second order derivatives. Calc.
Var. Partial Differential Equations 62 (2023), no. 4, Paper No. 118, 26 pp.

\bibitem {CFL22}Cazacu, C.; Flynn, J.; Lam, N. Sharp second order uncertainty
principles. J. Funct. Anal. 283 (2022), no. 10, Paper No. 109659, 37 pp.

\bibitem {CFL21}Cazacu, C.; Flynn, J.; Lam, N. Short proofs of refined sharp
Caffarelli-Kohn-Nirenberg inequalities. J. Differential Equations 302 (2021), 533--549.

\bibitem {CFLL23}Cazacu, C.; Flynn, J.; Lam, N.; Lu, G.
Caffarelli-Kohn-Nirenberg identities, inequalities and their stabilities.
arXiv:2211.14622, to appear in J. Math. Pures Appl.

\bibitem {DPP17} Devyver, B.; Pinchover, Y.; Psaradakis, G. Optimal Hardy
inequalities in cones. Proc. Roy. Soc. Edinburgh Sect. A 147 (2017), no. 1, 89--124.

\bibitem {DN21}Duong, A. T.; Nguyen, V. H. The sharp second order
Caffareli-Kohn-Nirenberg inequality and stability estimates for the sharp
second order uncertainty principle. arXiv preprint arXiv:2102.01425.

\bibitem {DLL22}Duy, N. T.; Lam, N.; Lu, G. $p$-Bessel pairs, Hardy's
identities and inequalities and Hardy-Sobolev inequalities with monomial
weights. J. Geom. Anal. 32 (2022), no. 4, Paper No. 109, 36 pp.

\bibitem {Fly20}Flynn, J. Sharp Caffareli-Kohn-Nirenberg-type inequalities on
Carnot groups. Adv. Nonlinear Stud. 20 (2020), no. 1, 95--111

\bibitem {FLL21}Flynn, J.; Lam, N.; Lu, G. Sharp Hardy identities and
inequalities on Carnot groups. Adv. Nonlinear Stud. 21 (2021), no. 2, 281--302.

\bibitem {FLL23}Flynn, J.; Lam, N.; Lu, G. $L^{p}$-Hardy identities and
inequalities with respect to the distance and mean distance to the boundary, arxiv.org

\bibitem {FS08}Frank, R. L.; Seiringer, R. Non-linear ground state
representations and sharp Hardy inequalities. J. Funct. Anal. 255 (2008), no.
12, 3407--3430.

\bibitem {GM}Ghoussoub, N.; Moradifam, A. Bessel pairs and optimal Hardy and
Hardy-Rellich inequalities. Math. Ann. 349 (2011), no.1, 1--57.

\bibitem {GM1}Ghoussoub, N.; Moradifam, A. Functional inequalities: new
perspectives and new applications. Mathematical Surveys and Monographs, 187.
American Mathematical Society, Providence, RI, 2013. xxiv+299.

\bibitem {KMP2007}Kufner, A.; Maligranda, L.; Persson, L.-E. The Hardy
Inequality. About its History and Some Related Results, Vydavatelsk\'{y}
Servis, Pilsen, 2007.

\bibitem {KP}Kufner, A.; Persson, L.-E. Weighted inequalities of Hardy type.
World Scientific Publishing Co., Inc., River Edge, NJ, 2003. xviii+357 pp.

\bibitem {LLZ19}Lam, N.; Lu, G.; Zhang, L. Factorizations and Hardy's type
identities and inequalities on upper half spaces. Calc. Var. Partial
Differential Equations 58 (2019), no. 6, Paper No. 183, 31 pp.

\bibitem {LLZ20}Lam, N.; Lu, G.; Zhang, L. Geometric Hardy's inequalities with
general distance functions. J. Funct. Anal. 279 (2020), no. 8, 108673, 35 pp.

\bibitem {LL01}Lieb, E.H.; Loss, M. Analysis. Second edition. Grad. Stud.
Math., 14 American Mathematical Society, Providence, RI, 2001. xxii+346 pp.

\bibitem {Maz11}Maz'ya, V. Sobolev spaces with applications to elliptic
partial differential equations. Second, revised and augmented edition.
Grundlehren der Mathematischen Wissenschaften [Fundamental Principles of
Mathematical Sciences], 342. Springer, Heidelberg, 2011. xxviii+866 pp.

\bibitem {OK}Opic, B.; Kufner, A. Hardy-type inequalities. Pitman Research
Notes in Mathematics Series, 219. Longman Scientific \& Technical, Harlow,
1990. xii+333 pp.

\bibitem {SteinWeiss}Stein, E. M.; Weiss, G. Introduction to Fourier analysis
on Euclidean spaces. Princeton Math. Ser., No. 32, Princeton University Press,
Princeton, NJ, 1971, x+297 pp.

\bibitem {TZ07}Tertikas, A.; Zographopoulos, N. B. Best constants in the
Hardy-Rellich inequalities and related improvements. Adv. Math. 209 (2007),
no. 2, 407--459.

\bibitem {Wang}Wang, J., $L^{p}$ Hardy's identities and inequalities for Dunkl
operators. Adv. Nonlinear Stud. 22 (2022), no. 1, 416-435.

\bibitem {Yaf99}Yafaev, D. Sharp constants in the Hardy-Rellich inequalities.
J. Funct. Anal. 168 , no. 1, 121--144.
\end{thebibliography}
\end{document}